\documentclass[twoside, 10pt, reqno]{article}

\usepackage{eucal}
\usepackage{amsmath}
\usepackage{amssymb}
\usepackage{theorem}
\usepackage[all,v2]{xy}
\xyoption{2cell}
\UseAllTwocells
\usepackage{graphicx}

\newtheorem{prop}{Proposition}[section]
\newtheorem{lem}[prop]{Lemma}

\newtheorem{cor}[prop]{Corollary}

\newtheorem{thm}[prop]{Theorem}

\theorembodyfont{\upshape}

\newtheorem{defn}[prop]{Definition}

\newtheorem{example}[prop]{Example}
\newtheorem{ex}[prop]{Example}

\newtheorem{rmk}[prop]{Remark}
\newtheorem{rem}[prop]{Remark}

\theoremstyle{remark}
\numberwithin{equation}{section}

            {\nolinebreak $\Box$ \end{trivlist}}

\newenvironment{proof}{\begin{trivlist}\item[]{\sc Proof.}}%
            { $\Box$ \end{trivlist}}

\newcommand{\eq}{\begin{eqnarray}}
\newcommand{\eneq}{\end{eqnarray}}
\newcommand{\eqn}{\begin{eqnarray*}}
\newcommand{\eneqn}{\end{eqnarray*}}

\newcommand{\bnum}{\begin{enumerate}[{\rm(i)}]}
\newcommand{\enum}{\end{enumerate}}
\newcommand{\banum}{\begin{enumerate}[{\rm(a)}]}
\newcommand{\eanum}{\end{enumerate}}

\newcommand{\noprint}[1]{}

\renewcommand{\tilde}{\widetilde}

\newcommand{\toto}{\rightrightarrows}

\newcommand{\f}{{\rm f}}

\newcommand{\X}{\mathop{\sf X}}
\newcommand{\XX}{{\mathfrak X}}
\renewcommand{\AA}{{\mathfrak A}}
\newcommand{\DD}{{\mathfrak D}}

\newcommand{\TT}{{\mathfrak T}}

\newcommand{\YY}{{\mathfrak Y}}
\newcommand{\UU}{{\mathfrak U}}

\newcommand{\ZZ}{{\mathfrak Z}}

\newcommand{\EE}{{\mathfrak E}}
\newcommand{\FF}{{\mathfrak F}}
\newcommand{\GG}{{\mathfrak G}}
\newcommand{\NN}{{\mathfrak N}}

\newcommand{\KK}{{\mathfrak K}}

\newcommand{\VV}{{\mathfrak V}}

\newcommand{\Aa}{{\mathfrak a}}
\newcommand{\Bb}{{\mathfrak b}}

\newcommand{\pr}{\mathop{\rm pr}\nolimits}

\newcommand{\Mor}{\mathop{\rm Mor}\nolimits}

\newcommand{\ob}{\mathop{\rm ob}}

\newcommand{\id}{\mathop{\rm id}\nolimits}
\newcommand{\Hom}{\mathop{\rm Hom}\nolimits}

\newcommand{\ldiag}[1]%
       {\makebox[0cm]{${\scriptstyle#1}\downarrow\phantom{\scriptstyle#1}$}}
\newcommand{\ldiagup}[1]%
       {\makebox[0cm]{${\scriptstyle#1}\uparrow\phantom{\scriptstyle#1}$}}
\newcommand{\rdiag}[1]%
       {\makebox[0cm]{$\phantom{\scriptstyle#1}\downarrow{\scriptstyle#1}$}}
\newcommand{\sediagr}[1]%
       {\makebox[0cm]{$\phantom{\scriptstyle#1}\searrow{\scriptstyle#1}$}}
\newcommand{\nediagr}[1]%
       {\makebox[0cm]{$\phantom{\scriptstyle#1}\nearrow{\scriptstyle#1}$}}
\newcommand{\rdiagup}[1]%
       {\makebox[0cm]{$\phantom{\scriptstyle#1}\uparrow{\scriptstyle#1}$}}
\newcommand{\swdiag}[1]%
       {\makebox[0cm]{$\phantom{\scriptstyle#1}\swarrow{\scriptstyle#1}$}}
\newcommand{\sediag}[1]%
       {\makebox[0cm]{${\scriptstyle#1}\searrow\phantom{\scriptstyle#1}$}}
\newcommand{\nediag}[1]%
       {\makebox[0cm]{${\scriptstyle#1}\nearrow\phantom{\scriptstyle#1}$}}

\newcommand{\doublearrowstack}[2]%
                      {{{{\scriptstyle#1}\atop{\textstyle\longrightarrow}}\atop{{\textstyle\longrightarrow}\atop{\scriptstyle#2}}}}
\newcommand{\rightleftarrowstack}[2]%
                      {{{{\scriptstyle#1}\atop{\textstyle\longrightarrow}}\atop{{\textstyle\longleftarrow}\atop{\scriptstyle#2}}}}
\newcommand{\leftrightarrowstack}[2]%
                      {{{{\scriptstyle#1}\atop{\textstyle\longleftarrow}}\atop{{\textstyle\longrightarrow}\atop{\scriptstyle#2}}}}

\newcommand{\Ra}{\Rightarrow}

\newcommand{\overtoparrow}%
{\makebox[0cm]{\beginpicture
\setcoordinatesystem units <.8cm,.4cm> point at 0 0
\setplotarea x from -3 to 3, y from 0 to 1
\setquadratic
\plot -3 0 0 1 3 0 /
\put{\vector(3,-1){0}}[Bl] at 3 0
\endpicture}}

\newcommand{\underbottomarrow}%
{\makebox[0cm]{\beginpicture
\setcoordinatesystem units <.8cm,.4cm> point at 0 0
\setplotarea x from -3 to 3, y from 0 to 1
\setquadratic
\plot -3 1 0 0 3 1 /
\put{\vector(3,1){0}}[Bl] at 3 1
\endpicture}}

\newcommand{\ses}[5]%
{0\longrightarrow#1\stackrel{#2}{ \longrightarrow}#3\stackrel{#4}{
\longrightarrow}#5\longrightarrow0}

\newcommand{\dt}[6]%
{#1\stackrel{#2}{longrightarrow}#3 \stackrel{#4}{\longrightarrow}#5
\stackrel{#6}{\longrightarrow} #1[1]}

\newcommand{\cat}[1]%
{(\mbox{\rm #1})}

\newcommand{\Top}{\mathsf{Top}}
\newcommand{\St}{\mathsf{St}}
\newcommand{\Fib}{\mathsf{Fib}}

\def\smashedlongrightarrow{\setbox0=\hbox{$\longrightarrow$}\ht0=1pt\box0}
\def\risom{\buildrel\sim\over{\smashedlongrightarrow}}

\def\smashedst{\setbox0=\hbox{$\rightrightarrows$}\ht0=4pt\box0}
\newcommand{\sst}[1]{\stackrel{#1}{\smashedst}}

\newcommand{\Map}{\operatorname{Map}}

\def \Lo {{\rm L}}

\newcommand{\lllra}[4]{\xymatrix@C=14pt{#1 \ar[r]^{#2}_*+[o][F-]{#3} & {#4}}}
\newcommand{\al}{\alpha}
\newcommand{\be}{\beta}

\newcommand{\BV}{{\bf BV}}
\def \LXX {{\rm L}\XX}

\def \ob {{\rm ob}}

\def \kor {{ k}}

\newcommand{\bbZ}{\mathbb{Z}}
\newcommand{\bbR}{\mathbb{R}}

\newcommand{\gy}[1]%
{ {{#1}^!} }

\newcommand{\cgy}[1]%
{ {{#1}_!} }

\def\smashedlongrightarrow{\setbox0=\hbox{$\longrightarrow$}\ht0=1pt\box0}
\def\risom{\buildrel\sim\over{\smashedlongrightarrow}}

\let\:=\colon
\let\oldcite\cite

\def \Y {\YY}

\def \X {\XX}

\def \K {\KK}
\def \Z {\ZZ}
\def \N {\NN}

\def \sz {(\mathbb{Z}/2\mathbb{Z})^{n+1}}

\newcommand{\sfT}{\mathsf{T}}

\newcommand{\arf}[5]{\ar@{}@<#2>#1 | (#3){#5}="#4" \ar#1  }
\newcommand{\ars}[3]{\arf{#1}{#2}{#3}{s}{}}
\newcommand{\art}[3]{\arf{#1}{#2}{#3}{t}{.}}

\newcommand{\CGTop}{\mathsf{CGTop}}
\newcommand{\Prin}{\mathfrak{Prin}}
\newcommand{\TrivPrin}{\mathfrak{TrivPrin}}
\newcommand{\Prinu}{\mathfrak{Prin}^{un}}
\newcommand{\Conj}{\operatorname{Conj}}
\newcommand{\pq}{q}
\newcommand{\mcG}{\mathcal{G}}

\begin{document}
\sloppy

\title{Group actions on stacks and applications to equivariant string topology for stacks}

\author{Gr\'egory Ginot\footnote{Sorbonne Universit\'es, UPMC Univ Paris 06, 
Institut de Math\'ematiques de Jussieu - Paris Rive Gauche,  UMR 7586, CNRS, Univ Paris Diderot, Sorbonne
Paris Cit\'e,  F-75005 Paris, France }
\and  Behrang~Noohi\footnote{School of Mathematical Sciences,
Queen Mary, Univ. of London, Mile End Road, 
London E1 4NS, United Kingdom}
 }

\maketitle

\begin{abstract} 
This paper is a continuations of the project initiated in~\cite{BGNX}. We construct string operations 
 on the $S^1$-equivariant homology of the (free) loop space $\Lo\XX$ of an oriented  differentiable 
 stack $\XX$ and show that $H^{S^1}_{*+\dim\XX -2}(\Lo\XX)$ is a graded Lie algebra. In the particular case where 
 $\XX$ is a 2-dimensional orbifold we give a Goldman-type description for the string bracket.
To prove these results, we develop a general machinery of (weak) group actions on topological stacks
which should be of independent interest. 
 We explicitly construct the quotient stack of a group acting on a stack
and show that it is a topological stack and further geometric  if $\XX$ is geometric. Then use its homotopy type to 
 define equivariant (co)homology for stacks, transfer maps, and so on.   
\end{abstract}

\tableofcontents

\section*{Introduction}
\addcontentsline{toc}{section}{Introduction}

One of the original motivations for the fundamental work of Chas-Sullivan in 
String Topology was to 
study the $S^1$-equivariant homology of the free loop space $LM=\Map(S^1,M)$ of a 
closed oriented manifold $M$. In particular, they showed in~\cite{CS} that 
this homology has a natural graded
Lie algebra structure which generalizes the classical Goldman bracket~\cite{Go}  
on free homotopy classes of loops on an oriented surface.  The main goal of this paper is 
to construct a similar graded
Lie algebra structure on the $S^1$-equivariant homology of the free 
loop stack $\Lo\XX=\Map(S^1,\XX)$ of an oriented differentiable stack. In particular,
we obtain  a generalization of the Goldman bracket to 2-dimensional orbifolds. Applied to
the case $\XX=[M/G]$, where $G$ is a (compact) Lie group acting on an oriented manifold
$M$, our construction gives rise to a $G$-equivaraint version of Chas-Sullivan's Lie algebra.

The following is one of the main results of this paper (see Corollary \ref{C:LieLoopStack} 
for a more precise statement
and Section \ref{SS:transfer} for the definition of the transfer map).

\begin{thm} 
  Let $\XX$ be an oriented Hurewicz stack of dimension $d$.
  Let $q:\LXX \to [S^1\backslash \LXX]$ be the projection map from the 
  loop stack to its quotient stack
  by the natural $S^1$-action. Let $T: H_*^{S^1}(\LXX)[2-d]\to H_*(\LXX)[1-d]$ 
  be the transfer map.   Then,
  for $x,y\in H_*(\LXX)$, the bracket defined by  the formula   
           $$\{x,y\}:= (-1)^{|x|} q\big(T(x)\star T(y)\big)$$  
makes the equivariant homology $H_*^{S^1}(\LXX)[2-d]$ into a graded Lie algebra. Furthermore,
the transfer map 
   $$T: H_*^{S^1}(\LXX)[2-d]\to H_*(\LXX)[1-d]$$ 
is a Lie algebra homomorphism. Here, $H_*(\LXX)[1-d]$ 
   is the Lie algebra structure underlying the \BV-algebra structure on $H_*(\LXX)$.
\end{thm}

The non-equivariant string topology for manifolds equipped with a $G$-action 
(or more generally for 
differentiable stacks, for instance orbifolds) has been studied by many authors 
(for example, in~\cite{LuUrXi} 
for finite groups $G$,  in~\cite{GrWe, CM, HL} for a Lie group acting trivially 
on a point, \cite{FT2} sor some global quotients 
and in our previous work~\cite{BGNX} for general oriented differentiable stacks). 
In~\cite{BGNX} we 
build a general setting allowing us to study string topology for stacks. In particular we define
\emph{functorial} loop stacks $L\mathfrak{X}=\mathop{Map}(S^1,\mathfrak{X})$, which are
again  topological stacks,  
and construct functorial $S^1$-actions on them.
 Contrary to the case of manifolds, constructing suitable\footnote{for instance models which have a well-defined 
 homotopy type and commutes with colimits of spaces when the target stack $\XX$ is differentiable} models  for 
 mapping stacks  is a nontrivial task, as the usual constructions using groupoids are rather complicated and not 
 functorial (nevertheless see~\cite{LuUrLoop} for a groupoid approach).  
 That is why we have chosen to work with stacks rather than groupoids (see~\cite{Mapping}). 

In~\cite{BGNX}, we  proved that the appropriately shifted homology of the 
free loop stack of an oriented stack is a Batalin-Vilkovisky
 algebra. Thereby, once we have the right tools to deal with the equivariant homology of stacks,  
 it is  possible to carry out Chas-Sullivan's original
  method for constructing the string bracket in the framework of stacks. 
This was the main motivation for the first part of this paper, in which we study the quotient of 
a (weak) action of a group $G$ on a (topological or differentiable) stack $\XX$, following the work of~\cite{Romagny}. 
This part is of independent interest and is expected to have applications beyond string topology; for instance it allows 
to define and study equivariant (co)homology for groups acting on stacks.  
Our main result in this part is the following 
(see Section \ref{SS:presentation} and Propositions \ref{P:toplogicalquotient}, \ref{P:differentiable}).

\begin{thm}
  Let $G$ be a topological (resp., Lie) group acting on a topological (resp., differentiable) stack $\XX$. 
  Then, there is a topological (resp., differentiable) stack $[G\backslash \XX]$
  together with  a map $\XX\to [G\backslash \XX]$ making $\XX$ into a $G$-torsor.  (Note that we do not need to define
  $[G\backslash \XX]$ as a 2-stack.)
\end{thm}

 This result allows us to define the $G$-equivariant
 (co)homology of a $G$-stack $\XX$ as the (co)homology of $[G\backslash \XX]$, in the same way that the $G$-equivariant
  (co)homology of a manifold is the (co)homology of the stack $[G\backslash M]$. In particular, we can apply the general 
 machinery of bivariant cohomology for stacks developed in~\cite{BGNX}, which allows  us to construct easily   
 Gysin (or \lq\lq{}umkehr\rq\rq{}) maps for equivariant cohomology (among other applications). In particular, we
 obtain transfer maps and the long homology 
 exact sequence relating the $S^1$-equivariant homology of an $S^1$-stack $\XX$ and its ordinary homology
 (Sections \ref{SS:transfer}, \ref{SS:GysinSequence}). This set of tools enable us to perform, 
 in a more or less formal manner, the standard constructions
 of manifold string topology in the more general setting of oriented stacks.

 One should note that even in the case of an orbifold $\XX$, 
 one has to consider the general theory of stacks to consider the free loop $\LXX$ with its $S^1$-action; further, the quotient 
 $[S^1\backslash \LXX]$ is indeed a stack but never an orbifold. 

\medskip
\noindent{\em Plan of the paper}

\smallskip

In Section \ref{S:Review} we review some basic results on topological stacks. 
We recall the notions of a {\em classifying space}
for a topological stack (which enables us to do algebraic topology),  
a {\em mapping stack} (which
enables us to define functorial loop stacks), and  a {\em bivariant theory} 
for topological stacks (which allows us to do
intersection theory, define Gysin maps, and so on). 

Sections \ref{S:Actions}--\ref{S:actionmapping} are devoted to the study of group actions on stacks. In Section
\ref{S:Quotient} we construct the quotient stack $[G\backslash \XX]$ of the action of a topological group $G$ 
on a topological stack $\XX$ and establish its main properties. We give two explicit
constructions for $[G\backslash \XX]$; one in terms of transformation groupoids, and one in
terms of torsors. We prove that $[G\backslash \XX]$ 
is always a topological stack, and that in the differentiable context it is a differentiable stack. In Section
\ref{S:EquivariantCoh} we use the results of Section \ref{S:Quotient}  to define the equivariant (co)homology
of a $G$-equivariant stack $\XX$. In Section \ref{S:actionmapping}  we focus on the case where $G$ is acting
on the mapping stack $\Map(G,\XX)$ by left multiplication.  In Section \ref{S:Homotopytype}, we
look at the homotopy type
of the unparameterized mapping stack $[G\backslash \Map(G,\XX)]$ .
The tools developed in the previous sections are robust enough to allow us to  carry forward standard constructions
in algebraic topology,
such as transfer maps and Gysin spectral sequence, to the stack setting in a more or less straightforward manner.
This is discussed is Section \ref{S:Transfer}.

In Section \ref{S:Stringproduct} we embark on proving the main result of the paper, namely the existence of a 
Lie algebra structure
on $H_*^{S^1}(\LXX)[2-\dim\XX]$.  We illustrate this result by looking at a few examples in Section \ref{SS:examples}. We refer to this Lie algebra structure as the 
\emph{string algebra} of $\XX$.

Our next goal is to study the case of a 2-dimensional oriented reduced orbifold $\XX$. This is done in Section \ref{S:Goldman}, 
where we give a
Goldman-type description for the Lie bracket of $H_0^{S^1}(\LXX)$. Observe that, as in the case of ordinary surfaces,
$H_0^{S^1}(\LXX)$ can be identified with the free module spanned by free homotopy classes of loops on $\XX$ (see 
Lemma \ref{L:surjectionpi1}). In Section \ref{SS:explicitGoldman} we give an algorithm for computing the Goldman
bracket on a reduced orbifold surface.

The main result that is used to prove that the Goldman-type bracket on $H_0^{S^1}(\LXX)$ coincides with 
the bracket constructed in Section \ref{S:Stringproduct} is the functoriality of the Batalin-Vilkovisky structure on 
$H_{*+\dim\XX}(\LXX)$
with respect to open embeddings. This is established in Section \ref{S:Lemma} and is a result which is 
interesting in its own right.

\medskip
\noindent{\em Further results}

\smallskip

In an upcoming paper we will study the Turaev cobracket and the coLie algebra structure on the equivariant homology of
the loop stack. We will also investigate the important role played by the ghost loops (the inertia stack). 

Moreover, slightly after a first version of our paper appears on the arxiv, a definition of string bracket for 
reduced 2-dimensional orbifolds obtained as quotient hyperbolic plane by fuchsian groups was given by Chas Gadgil~\cite{CG} purely in group theoretic terms. We will compare their 
bracket with ours in an upcoming paper. 

\medskip

\noindent{\bf Acknowledgement.} 
 Part of the research for this project was carried out at Max Planck Institute for Mathematics, Bonn.
The authors wish to thank MPIM for its support and  providing excellent  research conditions. The 
second author would like to also thank Fondation Sciences Math\'ematiques de Paris and \'Ecole Normale Sup\'erieure for their support while this project was initiated.
The first author was partially supported by the ANR grant GESAQ and ANR grant HOGT.
\vfill
\eject

\section{Notation and conventions}{\label{S:Notation}}

Throughout the notes, by a fibered groupoid we mean a category fibered in groupoids. 
We often identify a  space by the functor it represents and also by the corresponding fibered
category. We use the same notation for all.

When dealing with stack, by a fiber product we always mean a 2-fiber product.

We use multiplicative notation for composition of arrows in a groupoid. Namely, 
given composeable arrows
 $$\xymatrix@M=6pt{      
             x \ar[r]^{\alpha}   & y \ar[r]^{\beta}  & z   }$$      
their composition is denoted $\alpha\beta\: x \to z$.

If $V_*$ is a (homologically) graded $\kor$-module (or chain complex), 
we will denote by $V_*[1]$ its suspension, that is the graded $\kor$-module given by $(V_*[1])_i:=V_{i-1}$.

\section{Review of stacks}{\label{S:Review}}

In this section, we review some basic facts about stacks and fix some notation. For more details
on stacks the reader is referred to~\cite{Foundations}. For a quick introduction to stacks
which is in the spirit of this paper, the reader can consult~\cite{Rapid}.

Fix  a Grothendieck site $\sfT$ with a subcanonical topology (i.e.,
all representable functors are sheaves). Our favorite Grothendieck sites are 
$\Top$, the site of all topological spaces (with the open-cover topology), the site
$\CGTop$ of compactly generated topological spaces (with the open-cover topology) or the site $\mathsf{Diff}$ of smooth manifolds
(with the open cover topology).

A {\em stack} is a  fibered  groupoid $\XX$ over the site $\sfT$ satisfying the 
descent condition~\cite[\S~1.3]{Rapid}. Alternatively, we can use presheaves of groupoids 
instead of fibered groupoids, however, this is less practical for applications.

Stacks over $\sfT$ form a 2-category $\St_{\sfT}$ in which all 2-morphisms are isomorphisms. 
This is a full subcategory of the 2-category $\Fib_{\sfT}$
of fibered groupoids over $\sfT$.
A crucial property of the 2-category of fibered groupoids is that it has  2-fiber products.
The 2-fiber product is a fiberwise version of the following construction for groupoids.

Let $\XX$, $\YY$ and $\ZZ$ be groupoids and $f \: \YY \to \XX$ and $g \: \ZZ \to \XX$
functors. The {\em 2-fiber product}  $\Y\times_{\X}\Z$ is the groupoid which is defined as follows:
{\small

      $$\ob(\Y\times_{\X}\Z)=\{ 
              (y,z,\al) \ | \ y \in \ob\Y, z \in \ob\Z, \ 
                                      \al \: g(z) \to  f(y)  \ \text{an arrow in} \ \X
                                 \}$$
  $$\Mor_{\Y\times_{\X}\Z}\big( (y_1,z_1,\al),(y_2,z_2,\be) \big)=\left\{\begin{array}{rcl}
              (u,v) & | &  u\: y_1 \to y_2 ,\ v\: z_1 \to z_2 \ \text{s.t.:} \\
                    &   &   \hspace{0.5in}
         \xymatrix@R=10pt@C=12pt{
             f(y_1) \ar[r]^{f(u)}  \ar@{}[rd]|{\circlearrowleft}
                                                   &  f(y_2) \\
                            g(z_1)  \ar[u]^{\al}  \ar[r]_{g(u)}  &  g(z_2) \ar[u]_{\be}   }
                                                          \end{array}\right\}$$}

The 2-category of stacks is closed under 2-fiber products.
Since we will never use the strict fiber product of groupoids in this paper, we will often refer to the
2-fiber product as fiber product. 

To every object $T$ in $\sfT$ we associate a fibered groupoid by applying the Grothendieck
construction to te functor it represents. We use the same notation $T$ for this fibered groupoid.
This induced a functor from $\sfT$ to the 2-category of stacks over $\sfT$. This functor 
is fully faithful thanks to the following lemma. So there is no loss of information in regarding $T$ as 
a stack over $\sfT$.

\begin{lem}[Yoneda lemma]{\label{L:Yoneda}}
  Let $\X$ be a category fibered in groupoids over $\sfT$, and let $T$
  be an object in $\sfT$. Then, the natural functor
    $$\Hom_{\Fib_{\sfT}}(T,\X) \to \X(T)$$
  is an equivalence of groupoids.
\end{lem}

A morphism $f \: \XX \to \YY$ of fibered groupoids is called an {\em epimorphism}
if for every object $T$ in $\sfT$, every $y \in \YY(T)$ can be lifted, up to isomorphism, 
to some $x \in \XX(T)$, possibly after replacing $T$ with an open cover.
For example, in the case where $\XX$ and $\YY$ are honest topological spaces,
$f \: \XX \to \YY$ is an epimorphism if and only if it admits local sections. 

The inclusion $\St_{\sfT} \to \Fib_{\sfT}$ of the 2-category of stacks in the 2-category of
fibered groupoids has a left adjoint which is called the {\em stackification functor}
and is denoted by $\XX \mapsto \XX^+$. There is a natural map
$\XX \to \XX^+$ (the unit of adjunction). The naturality means that we have a natural 
2-commutative diagram
  $$\xymatrix@M=6pt{\XX \ar[d] \ar[r]^f & \XX \ar[d] \\
     \XX^+ \ar[r]_f & \YY^+}$$

(In fact, the stackification functor can be constructed in way that the above diagram is
strictly commutative, but we do not really need this here.) The 
construction of the stackification 
functor involves taking limits and filtered colimits, hence it commutes with 2-fiber products.

Now let $\sfT$ be a category of topological spaces. A morphism
$f \: \XX \to \YY$ of stacks over $\sfT$ is called {\em representable}, if for every
morphism $T \to \YY$ from a topological space $T$, the fiber product $T\times_{\XX} \YY$
is equivalent to a topological space.

To any topological groupoid $[R\toto X]$ in $\sfT$ we can associate its {\em quotient
stack} (see~\cite[Definition 3.3]{Foundations}, or \cite[\S~1.5]{Rapid}). 
 A stack $\XX$ over $\sfT$ is a \textbf{topological stack} if and only if the following two
 conditions are satisfied~\cite[\S~7]{Foundations}:
 \begin{itemize}
   \item The diagonal $\XX \to \XX\times \XX$ is representable;
   \item There exists a topological space $X$  (called an {\em atlas}
   for $\XX$) and an epimorphism $p \: X \to \XX$.
 \end{itemize}
The first condition is indeed equivalent to all morphisms $T \to \XX$ from a topological
space $T$ to $\XX$ being representable. 
Given an atlas $X \to \XX$, we obtain a groupoid presentation $[R \toto X]$ for
$\XX$, where $R=X\times_{\XX} X$ and the source and target
maps $s,t \: R \to X$ are the projection maps.

In this paper we will mainly be interested in {\bf Hurewicz stacks}. These are topological 
stacks for which there exists a groupoid presentation $[R\toto X]$ with the source (hence also 
target) map $s \: R \to X$ a local Hurewicz fibration. The latter means that every point
in $R$ has a neighborhood $U$ such that $s|_U$ is a Hurewicz fibration onto its image.

Recall that a {\bf differentiable stack} is a stack over the site 
$\sfT=\mathsf{Diff}$  of $C^\infty$-manifolds that admits a 
 groupoid presentation $[R\toto X]$ with
the source map $s \: R \to X$ a smooth submersion of differentiable manifolds (that is,
$[R\toto X]$ is a Lie groupoid). 
Applying the  forgetful functor 
$\mathsf{Diff} \to \mathsf{CGTop}$ to (a groupoid presentation of) a differentiable 
stack $\XX$ we obtain a topological stack which we view as  the underlying topological stack
of $\XX$.  This topological stack is automatically
Hurewicz.

\subsection{Classifying spaces of topological stacks}{\label{SS:classifying}}

Incorporating ideas of  Haefliger~\cite{Haefliger} and Segal~\cite{Segal} on 
classifying spaces of categories, 
and  fibration techniques developed by Dold~\cite{Dold}, in~\cite{Homotopytypes} the second 
author has developed a machinery of
classifying spaces for topological stacks which is particularly suitable for our applications in
this paper. In this section we recall some of the basic facts from~\cite{Homotopytypes}.

Let $\XX$ be a topological stack. By a {\em classifying space}
for $\XX$ \cite[Definition 6.2]{Homotopytypes}, we mean a topological
space $X$ and a morphism $\varphi \: X \to \X$ which is a {\em universal weak equivalence}.
The latter means that, for every  map $T \to \X$ from
a  topological space $T$, the base extension  $\varphi_T \: X_T \to T$ is a weak equivalence of
 topological spaces

\begin{thm}[\oldcite{Homotopytypes}, Theorem 6.3]{\label{T:nice}}
 Every topological stack $\X$ admits an atlas $\varphi \: X \to \X$
  with the following property. For every map $T \to \X$ from
  a paracompact topological space $T$, the base extension
  $\varphi_T \: X_T \to T$ is shrinkable map of
  topological spaces, in the sense that, it admits a
  a section $s \: T \to X_t$ and a fiberwise deformation retraction
  of $X_T$ onto $s(T)$. 
   In particular, $\varphi \: X \to \X$ makes $X$ a classifying space for $\XX$.
\end{thm}

The fact that $\varphi \: X \to \X$ is universal weak equivalence essentially means that
we can identify the homotopy theoretic information in $\XX$ and $X$ via $\varphi$. 

The classifying space is unique up to a unique (in the weak homotopy category)
weak equivalence. In the case where $\X$ is the quotient stack $[G\backslash M]$ 
of a group action,
it can be shown~\cite[\S~4.3]{Homotopytypes} that the Borel construction $M\times_GEG$ is a
classifying space for $\X$. Here, $EG$ is the universal $G$-bundle in the sense of Milnor.

Classifying spaces can be used to define homotopical invariants for topological 
stack~\cite[\S~11]{Homotopytypes}. For example, to define the relative homology
of a pair $\AA\subset \XX$, we choose a classifying space $\varphi \: X \to \XX$
and define $H_*(\XX,\AA):=H_*(X,\varphi^{-1}\AA)$. The fact that
$\varphi$ is a universal weak equivalence guarantees that 
this is well defined up to a canonical
isomorphism. In the case where  $\X$ is the quotient stack $[G\backslash M]$
 of a group action and $\AA=[G\backslash A]$ is the quotient
of a $G$-equivariant subset $A$ of $M$,
this gives us the $G$-equivariant homology  of the pair $(M,A)$ (as defined via the Borel 
construction).

\subsection{Bivariant theory for stacks}\label{SS:bivarianttheory}
In~\cite{BGNX} we showed that the (singular) (co)homology for topological stacks extends to a (generalized) 
\emph{bivariant theory} \`a la Fulton-Mac Pherson~\cite{FM}. In fact, 
in~\cite{BGNX}, we associate, to any map 
$f:\XX\to \YY$ of stacks, graded $k$-modules $H^\bullet(\XX\stackrel{f}\to \YY)$ (called the bivariant homology 
group of $f$) such that $H^{\bullet}(X\stackrel{id}\to \XX)$ is the singular cohomology of $\XX$ and similarly, 
the homology groups of $\XX$ are given by  $H_n(\XX) \cong H^{-n}(\XX\to pt)$.

This bivariant theory is endowed with three kinds of operations:
\begin{itemize}
 \item \textit{(composition)} or products generalizing the cup-product;
 \item \textit{(pushforward)} generalizing the homology pushforward;
 \item \textit{(pullback)} generalizing the pullback maps in cohomology.
\end{itemize}
These operations satisfy various compatibilities and allow us to build Poincar\'e duality, Gysin and transfer homomorphisms easily.

\subsection{Mapping stacks}{\label{SS:mapping}}

We begin by recall  the definition and the main properties
of mapping stacks. For more details see~\cite{Mapping}.

Let $\X$
and $\Y$ be   stacks over $\sfT$. We define the stack
$\Map(\Y,\X)$, called the {\bf mapping stack} from $\Y$ to $\X$, by
the defining its groupoid of section over $T \in \sfT$ to be 
$\Hom(\Y\times T,\X)$,
   $$\Map(\Y,\X)(T)=\Hom(\Y\times T,\X),$$
where $\Hom$ denotes the groupoid of stack morphisms. This is easily
seen to be a stack.

We have a natural equivalence of groupoids
    $$\Map(\Y,\X)(*)\cong\Hom(\Y,\X),$$
where $*$ is a point. In particular, the underlying set of the coarse 
moduli space of 
$\Map(\Y,\X)$ is the set of 2-isomorphism classes of morphisms
from $\Y$ to $\X$.

If $\sfT=\CGTop$, it follows from the exponential law for mapping
spaces that when $X$ and $Y$ are spaces, then
$\Map(Y,X)$ is representable by the usual mapping space from $Y$ to
$X$ (endowed with the compact-open topology).

The mapping stacks are functorial in both variables.

\begin{lem}{\label{L:functorial}}
 The mapping stacks $\Map(\Y,\X)$ are functorial in $\X$ and
 $\Y$.
 That is, we have natural functors
 $\Map(\Y,-) \: \St \to \St$ and $\Map(-,\X) \: \St^{op} \to
 \St$. Here, $\St$ stands for the 2-category of stacks over
 $\sfT$ and  $\St^{op}$ is the opposite category (obtained by inverting the direction of
 1-morphisms in $\St$).
\end{lem}

The exponential law holds for mapping stacks.

\begin{lem}{\label{L:exponential}}
 For stacks $\X$, $\Y$ and $\Z$ we have a natural equivalence of stacks
   $$\Map(\Z\times\Y,\X)\cong\Map(\Z,\Map(\Y,\X)\big).$$
\end{lem}

For the following theorem to be true we need to assume that $\sfT=\CGTop$.

\begin{thm}[\cite{Mapping}, Theorem 4.2]{\label{T:mapping}}
  Let $\X$ and $\K$ be  topological stacks. 
  Assume that
  $\K\cong[K_0/K_1]$, where $[K_1\sst{}K_0]$ is a topological
  groupoid with $K_0$ and $K_1$ compact.
  Then, $\Map(\K,\X)$ is a topological stack.
\end{thm}

We define  the {\em free loop stack} of a stack $\X$ to be $\Lo\X:=\Map(S^1,\X)$. 
If $\X$ is a  topological stack, then it follows from the above theorem that  $\Lo\X$ is also a 
topological stack. 

Theorem \ref{T:mapping} does not seem to be true in general without the compactness condition on the
$\K$. However, it is good to keep in mind the following general fact~\cite[Lemma 4.1]{Mapping}.

\begin{prop}{\label{L:repdiag}}
     Let $\X$ and $\Y$ be  topological stacks.  Then,
   for every topological space $T$, every morphism $T \to
   \Map(\Y,\X)$ is representable. (Equivalently,  
   $\Map(\Y,\X)$ has a representable diagonal.)
\end{prop}

The following result is useful in computing homotopy types of mapping stacks.

\begin{thm}[\oldcite{Mapping}, Corollary 6.5]{\label{T:invariance}}
   Let $Y$ be a paracompact topological space and $\X$ a topological stack. Let
   $X$ be a classifying space for $\X$  with $\varphi \: X \to \X$ as in Theorem \ref{T:nice}. 
   Then,  the induced map
   $\Map(Y,X) \to \Map(Y,\X)$ makes $\Map(Y,X)$ a classifying space for 
   $ \Map(Y,\X)$ (in particular,  it 
   is a universal weak equivalence).
\end{thm}

\begin{cor}[\oldcite{Mapping}, Corollary 6.6]{\label{C:invariance}}
   Let $\X$ be a topological stack and $X$ a classifying space for it, 
   with $\varphi \: X \to \X$ as in Theorem \ref{T:nice}. Then, the 
   induced map $L\varphi \: \Lo X \to \LXX$ on free 
   loop spaces makes $LX$   a classifying space for $\LXX$ 
   (in particular, $L\varphi$
   is a universal weak equivalence).
\end{cor}

\section{Group actions on stacks}{\label{S:Actions}}

\subsection{Definition of a group action}{\label{SS:definition}}

In this subsection we recall the definition of a weak group action 
on a groupoid from~\cite{Romagny}.
This definition is more general than what is needed
for our application  (\S~\ref{SS:S1actionloop}) as in our case the action 
will be strict (i.e., the transformations
$\al$ and $\Aa$ in Definition \ref{D:action} will be identity transformations).

Let $\XX$ be a  fibered  groupoid
over $\sfT$ and $G$ a group over $\sfT$. (More generally, we can take $\XX$ to be an
object and $G$ a strict monoid object in any fixed 2-category.)

\begin{defn}[\cite{Romagny}, 1.3(i)]{\label{D:action}}
   A left {\bf action} of $G$ on $\XX$ is a triple $(\mu,\al,\Aa)$ where $\mu \: G\times\XX 
   \to \XX$
   is a morphism (of fibered groupoids), and $\al$ and $\Aa$ are 2-morphisms as in the
   diagrams
        $$\xymatrix@R=22pt@C=38pt@M=5pt{  G\times G  \times \XX 
               \art{[r]}{-1ex}{0.4}  ^(0.55){m\times\id_{\XX}}  
               \ars{[d]}{1ex}{0.4} _{\id_G\times \mu} 
                                  & G \times  \XX \ar[d]^{\mu}    \ar  @/_/ @{=>}_{\al} "s";"t"  \\
                                       G \times \XX \ar[r]_{\mu}        &   \XX         }  \ \ \ \ \    
            \xymatrix@R=22pt@C=38pt@M=5pt{ G \times \XX 
                \art{[r]}{-1ex}{0.37}  ^{\mu}   & \XX   \\
                                \XX    \ars{[ru]}{+1ex}{0.5} _{\id_{\XX}} 
                                 \ar[u]^{1\times  \id_{\XX} }    \ar   @{=>}^(0.3){\Aa} "s";"t"  &    }$$                                                                
    We require the following equalities:
    \begin{itemize}
       \item[A1)] $(g\cdot \al_{h,k}^x)\al_{g,hk}^x=\al_{g,h}^{k\cdot x}\al_{gh,k}^x$, 
       for all $g,h,k$ in $G$  and $x$ an object in $\XX$. 
       \item[A2)] $(g\cdot\Aa^x)\al_{g,1}^x=1_{g\cdot x}=\Aa^{g\cdot x}\al_{1,g}^x$, 
       for every $g$ in $G$
       and $x$ an object in $\XX$. 
    \end{itemize}   
    The dot in the above formulas is a short for the multiplication $\mu$. 
    Also, $\al_{g,h}^x$ stands for the arrow $\al(x,g,h)\: g\cdot(h \cdot x) \to 
   (gh) \cdot  x$ in $\XX$ and $\Aa^x$ for the arrow
    $\Aa(x) \: x \to 1\cdot  x$.  
\end{defn}

Let $\XX$ and $\YY$ be  fibered   groupoids 
over $\sfT$ endowed with an action of $G$ as in definition \ref{D:action}.

\begin{defn}[\cite{Romagny}, 1.3(ii)]{\label{D:morphism}}
    A $G$-{\bf equivariant} morphism between $(\XX,\mu,\al,\Aa)$ and 
    $(\YY,\nu,\be,\Bb)$ is a morphism $F \: \XX \to \YY$
    together with a 2-morphism $\sigma$ as in the diagram
      $$\xymatrix@R=22pt@C=30pt@M=5pt{G \times \XX 
               \ars{[r]}{-1ex}{0.4}  ^(0.55){\mu}  
               \art{[d]}{1ex}{0.4} _{\id_G\times F} 
                                  & \XX  \ar[d]^{F}    \ar  @/^/ @{=>}^{\sigma} "s";"t"  \\
                                     G\times  \YY  \ar[r]_{\nu}        &   \YY         } $$
     such that 
     \begin{itemize}
        \item[B1)] $\sigma_g^{h\cdot x}(g\cdot\sigma_h^x)\be_{g,h}^{F(x)}=
        F(\al_{g,h}^x)\sigma_{gh}^x$,
           for every $g,h$ in $G$ and $x$ an object in $\XX$.
        \item[B2)] $F(\Aa^x)\sigma_1^x=\Bb^{F(x)}$, for every object $x$ in $\XX$.  
     \end{itemize}
     Here, $\sigma^x_g$ stands for the arrow $\sigma(x,g) \:   F(g\cdot x)
     \to g\cdot F(x)$ in $\YY$.
     We often drop $\sigma$ from the notation and denote such a morphism simply by $F$.
\end{defn}

\begin{defn}{\label{D:2morphism}}
  Let $(F,\sigma)$ and $(F',\sigma')$ be $G$-equivariant morphisms from $\XX$ to $\YY$
  as in Definition \ref{D:morphism}. A {\bf $G$-equivariant 2-morphism} from 
  $(F,\sigma)$ to $(F',\sigma')$ is a 2-morphism $\varphi \: F \Ra F'$ such that
    \begin{itemize}
         \item[C1)]  $(\sigma_g^x)(g\cdot \varphi_x)=(\varphi_{g\cdot x})(\sigma'_g{}^{x})$, 
         for every $g$ in $G$ and $x$ an object in $\XX$.  
    \end{itemize}
  Here, $\varphi_x \: F(x) \to F'(x)$ stands for the effect of $\varphi$ on $x \in \ob\XX$.  
\end{defn}

\subsection{Transformation groupoid of a group action}{\label{SS:transformationgpd}}

Suppose now that $G$ is a discrete group and $\XX$ a groupoid (i.e., the base category $\sfT$
is just a point). Given a group action $\mu \: G\times \XX  
   \to \XX$ as in Definition \ref{D:action}, we define the transformation groupoid 
   $[G\backslash\XX]$ as follows.
   The objects of $[G\backslash\XX]$ are the same as those of $\X$,
       $$\ob[G\backslash\XX]=\ob\X.$$
   The morphisms of    $[G\backslash\XX]$ are 
     $$\Mor [G\backslash\XX]=\{(\gamma,g,x) \ | \ y \in \ob\X, \ g \in G, \ \gamma \in 
     \Mor\X, \ t(\gamma)=g\cdot y\}.$$  
    We visualize the arrow $(\gamma,g,x)$ as follows:
              $$\xymatrix@M=6pt{    &   y   \\   
             x   \ar[r]^{\gamma}    &  g\cdot y \ar@{..>}[u]_{g}   }$$  
   The source and target maps are defined by
     $$s(\gamma,g,y)=s(\gamma)=x \ \text{and}\  \ t(\gamma,g,y)=y.$$
   The composition of arrows is defined by
     $$(\gamma,g,y)(\delta,h,z)=\big(\gamma(g\cdot \delta)\al_{g,h}^z,\,gh,\,z\big).$$
   The identity morphism of an object $x \in \ob[G\backslash\XX]=\ob\XX$ is
       $$(\Aa^x,1,x).$$
   Pictorially, this is
          $$\xymatrix@M=6pt{  &   x   \\   
             x \ar[r]^{\Aa^x}    &    1\cdot x \ar@{..>}[u]_1  }$$
          
   Finally, the inverse of an arrow $(\gamma,g,y)$ in $[G\backslash\XX]$ is given by
     $$\big(\Aa^y(\al_{g^{-1},g}^y)^{-1}(g^{-1}\cdot\gamma^{-1}),\, g^{-1},\, x\big),$$
  where $x=s(\gamma)$.
  
It follows from the axioms (A1) and (A2) of Definition
\ref{D:action} that the above definition  makes $[G\backslash\XX]$ into a
groupoid. In fact,  axioms (A1) and (A2) are equivalent to $[G\backslash\XX]$ being a groupoid.

There is a natural functor $\pq \: \XX \to [G\backslash\XX]$ which is the identity  on the objects,
that is, $\pq(x)=x$. On arrows it is defined by 
     $$\pq(\gamma)=\big(\gamma\Aa^y,\, 1,\, y\big),$$    
where $y=t(\gamma)$. Pictorially, this is
      $$\xymatrix@M=6pt{     &   &  y \\   
             x \ar[r]^{\gamma}   & y \ar[r]^{\Aa^y}  & 1\cdot y \ar@{..>}[u]_1  }$$      
The functor $\pq$ is faithful.

Since $\pq \: \XX \to [G\backslash\XX]$ is faithful, we can regard $\XX$ as a subcategory of $[G\backslash\XX]$.
For this reason, we will often denote $\pq(\gamma)$ simply by $\gamma$, if there is no fear
of confusion. We also use the short hand notation $g^y$ for the arrow $(1_{g\cdot y},\, g,\, y)$.
This way, we can write $(\gamma,g,y)=\gamma g^y$. 

The groupoid $[G\backslash\XX]$ can be defined,
alternatively, as the groupoid generated by $\XX$ and the additional arrows $g^x$ subject
to certain relations which we will not spell out here. It is important, however, to observe the
following commutativity relation $g^x\gamma=(g\cdot \gamma)g^y $, as in the following
commutative diagram:
        $$\xymatrix@M=6pt{ x   \ar[r]^{\gamma}
         &   y   \\   
             g\cdot x   \ar[u]^{g^x}  \ar[r]_{g\cdot \gamma}
         & g\cdot  y  \ar[u]_{g^y}   }$$
         
Yet another way to define the groupoid $[G\backslash\XX]$ is to define it as the groupoid of trivialized $G$-torsors $P$, endowed with
a $G$-equivariant map $\chi \: P \to \X$ which satisfies $\chi(g)=g\cdot\chi(1)$, for every $g\in P$. 
Here, by a trivialized $G$-torsors $P$ we mean $P=G$ viewed as
a $G$-torsor via left multiplication.  

This definition gives a groupoid that is isomorphic to the one defined above. It also explains
our rather unnatural looking convention of having the arrow $g^x$ go from $g\cdot x$ to $x$ rather than other way around.  
If we drop the extra condition  $\chi(g)=g\cdot\chi(1)$ in the definition, we get a groupoid which is naturally equivalent 
(but not isomorphic) to $[G\backslash\XX]$.
For more on torsors 
see \S~\ref{SS:torsors}.

\begin{rem}{\label{R:strict}}
  In the case where the action of $G$ on $\XX$ is strict, an arrow $(\gamma,g,x)$ in
  $[G\backslash\XX]$ is uniquely determined by $(\gamma,g)$, i.e., $x$ is redundant. When
  $\XX$ is a set, $[G\backslash\XX]$ is equal to the usual transformation groupoid of the action
  of a group on  a set.
\end{rem}

\begin{example}{\label{E:crossproduct}}
 Let $\XX$ be a groupoid with one object, and let $H$ be its group of morphisms.
 Suppose that
 we are given a strict action of a group $G$ on $\XX$ (this amounts to an action of $G$ on 
 $H$ by homomorphisms). Then, $[G\backslash\XX]$ is the groupoid with one object whose
 group of morphisms is $H\rtimes G$. In other words, $[G\backslash BH]=B(H\rtimes G)$.

\end{example}

Given a $G$-equivariant morphism $F$ as in Definition \ref{D:morphism}, we obtain a functor
  $$[F] \: [G\backslash\XX] \to [G\backslash\YY]$$
 as follows. The effect of $[F]$ on objects is the same as that of $F$, i.e., $[F](x):=F(x)$. 
 For a morphism
 $(\gamma,g,y)$ in $ [G\backslash\XX] $ we define
   $$[F](\gamma,g,y):=\big(F(\gamma)\sigma_g^y,\, g,\, F(y)\big).$$

It follows from the axioms (B1) and (B2) of Definition \ref{D:morphism} that $[F]$ is a
functor. In fact, axioms (B1) and (B2) are equivalent to $[F]$ being a functor. Furthermore,
the  diagram 
      $$\xymatrix@M=6pt{\XX
           \ar[d]_{\pq_{\XX}} \ar[r]^{F}   &   
              \YY \ar[d]^{\pq_{\YY}}  \\   
         [G\backslash\XX]    \ar[r]_{[F]}  
        &    [G\backslash\YY]    }$$
is 2-cartesian and strictly commutative.

Given a $G$-equivariant 2-morphism $\varphi$ as in Definition \ref{D:2morphism}, we
obtain a natural transformation of functors $[\varphi] \: [F] \Ra [F']$ whose effect on
$x \in \ob\XX$ is defined by $[\varphi](x):=\varphi(x)  \: F(x) \to F'(x)$.
It follows from the axiom (C1) of Definition \ref{D:2morphism} that $[\varphi]$ is a
natural transformation of functors. In fact, axiom (C1) is equivalent to $[\varphi]$ being a 
natural transformation of functors.  

\subsection{The main property of the transformation groupoid}{\label{SS:properties}}

The most important property of the transformation groupoids for us is the fact that the
diagram 
        $$\xymatrix@M=6pt{G\times \XX
           \ar[d]_{\pr_2} \ar[r]^{\mu}   &   
              \XX \ar[d]^{\pq}  \\   
         \XX    \ar[r]_{\pq}  
        &   [G\backslash\XX]    }$$
is 2-cartesian. In other words, the functor
  $$(\pr_2,\mu) \:G\times \XX \to \XX\times_{[G\backslash\XX]}\XX$$ 
 is an equivalence of groupoids. This is an easy verification and we leave it to the reader.
 
 \begin{lem}{\label{L:projection}}
     Let $T$ be a set (viewed as a groupoid with only identity morphisms) and $f \: T \to 
     [G\backslash\XX]$ a functor. 
    Then, the groupoid $T\times_{[G\backslash\XX]}\XX$ is equivalent to a set. 
    If we denote the set of isomorphism classes of $T\times_{[G\backslash\XX]}\XX$  by $P$,
    then the natural left $G$-action on $P$ (induced from the action of $G$ on the second factor 
    of fiber product) makes $P$ a left $G$-torsor.
 \end{lem}

\begin{proof}
  This is a simple exercise (e.g., using the above 
  2-cartesian square). 
\end{proof}

\section{Quotient stack of a group action}{\label{S:Quotient}}

In this section we study the global version of the construction of the transformation groupoid
introduced in \S~\ref{SS:transformationgpd} and use it to define the quotient stack of 
a weak group action on a stack. We fix a Grothendieck site $\sfT$ through this section.
The reader may assume that $\sfT$ is the site $\Top$ of all topological spaces, or the site
$\CGTop$ of compactly generated topological spaces.

\subsection{Definition of the quotient stack}{\label{SS:quotient}}

Let $\XX$ be a  fibered groupoid over $\sfT$ and $G$ a presheaf of groupoids over $\sfT$ 
viewed as a fibered groupoid). Suppose that we have a right action  $\mu \: G\times \XX\to \X$
of $G$ on $\XX$ as in Definition \ref{D:action}. Repeating the construction of the
transformation groupoid as in \S~\ref{SS:transformationgpd}, we obtain a a category fibered
in groupoids $\lfloor G\backslash\XX \rfloor$. (The reason for not using the square brackets 
becomes clear shortly.)
In terms of section, $\lfloor G\backslash\XX \rfloor$ is determined by the following property:
  $$\lfloor G\backslash\XX \rfloor(T)=[G(T)\backslash\XX(T)], \ \text{for every} \ T \in \sfT.$$

The following lemma is straightforward.

\begin{lem}{\label{L:prestack}}
  Notation being as above, if $\XX$ is a prestack and $G$ is a sheaf of groups, then
  $\lfloor G\backslash\XX \rfloor$ is a prestack.
\end{lem}

If in the above lemma $\XX$ is a stack, it is not necessarily true that $\lfloor G\backslash\XX \rfloor$ is
a stack (this is already evident in the case where $\XX$ is a sheaf of sets). Therefore, we make
the following definition.

\begin{defn}{\label{D:quotientstack}}
  Let $\XX$ be a stack and $G$ acting on $\XX$ (Definition \ref{D:action}). We define
  $[G\backslash\XX]$ to be the stackification of the prestack $\lfloor G\backslash\XX \rfloor$.
\end{defn}

There is a natural epimorphism of stacks $\pq \: \XX \to [G\backslash\XX]$. This morphism is strictly functorial,
in the sense that, for every $G$-equivariant morphism (Definition \ref{D:morphism}) $F \: \XX \to \YY$
of $G$-stacks, there is a natural induced morphism 
$[F] \: [G\backslash\XX] \to [G\backslash\YY]$ of stack such that
 the diagram
        \begin{equation}\label{eq:quotientequivcartesian} \xymatrix@M=6pt{\XX
           \ar[d]_{\pq_{\XX}} \ar[r]^{F}   &   
              \YY \ar[d]^{\pq_{\YY}}  \\   
         [G\backslash\XX]    \ar[r]_{[F]}  
        &    [G\backslash\YY]  }\end{equation}
 is 2-cartesian and strictly commutative. 
 This follows from the corresponding statement in the discrete case
 (see end of \S~\ref{SS:transformationgpd}) and the similar properties of the stackification functor 
 \S~\ref{S:Review}. Similarly, given a $G$-equivariant 2-morphism $\varphi
 \: F \Ra F'$, we
obtain a 2-morphism $[\varphi] \: [F] \Ra [F']$.

Since the stackification functor commutes with 2-fiber products, we have a 2-cartesian
square (see \S~\ref{SS:properties})
          $$\xymatrix@M=6pt{G\times \XX 
           \ar[d]_{\pr_2} \ar[r]^{\mu}   &   
              \XX \ar[d]^{\pq}  \\   
         \XX    \ar[r]_{\pq}  
        &   [G\backslash\XX]    }$$
and the functor
  $$(\pr_2,\mu) \:\XX\times G \to \XX\times_{[G\backslash\XX]}\XX$$ 
 is an equivalence of stacks.  Here $\mu \: G\times  \XX \to \XX$ stands for the action of
 $G$ on $\XX$.
 
 \begin{lem}{\label{L:representableproj}}
     Let $\XX$ be a stack with a group action as above. 
     Let $T$ be a sheaf of sets (viewed as a fibered groupoid over $\sfT$) and $f \: T \to 
     [G\backslash\XX]$ a morphism. 
    Then, the stack $T\times_{[G\backslash\XX]}\XX$ is equivalent to the sheaf of sets $P$,
    where $P$ is the sheaf of isomorphism classes of $T\times_{[G\backslash\XX]}\XX$.   
    Furthermore,
    the natural left $G$-action on $P$ (induced from the action of $G$ on the second factor 
    of fiber product) makes $P$ a left $G$-torsor.
 \end{lem}

 \begin{proof}
   It follows from Lemma \ref{L:projection} that $T\times_{[G\backslash\XX]}\XX$ is (equivalent to)
   a presheaf of sets, namely $P$. On the other hand, since stacks are closed under fiber product,
   $T\times_{[G\backslash\XX]}\XX$ is a  stack. Hence, it is (equivalent to) a sheaf of sets. Thus, 
   $P$ is indeed a sheaf. It follows from Lemma \ref{L:projection} that $P$ is a left $G$-torsor.
 \end{proof}

\subsection{Interpretation in terms of torsors}{\label{SS:torsors}}

Let $\XX$ be a stack with an action of a sheaf of groups $G$. Let $T$ be an object in $\sfT$.
We define the groupoid $\Prin_{G,\XX}(T)$ as follows. 
 {\small   $$\ob\Prin_{G,\XX}(T)=\left\{\begin{array}{rcl}
              (P,\chi) & | &  P\to T \ \ \text{left $G$-torsor}, \\
                    &   &    \chi \: P \to \XX \ \
              \text{$G$-equivariant map}   
                                                          \end{array}\right\}$$
      $$\Mor_{\Prin_{G,\XX}(T)}\big((P,\chi),(P',\chi')\big)=\left\{\begin{array}{rcl}
              (u,\phi) & | &  u \: P \to P'  \ \text{map of $G$-torsors}, \\
                    &   &    \phi \: \chi \Ra \chi'\circ u \ \ \text{$G$-equivariant}
                                                          \end{array}\right\}$$  }                                                      
 The groupoid $\Prin_{G,\XX}(T)$ contains a full subgroupoid 
 $\TrivPrin_{G,\XX}(T)$ consisting of those pairs $(P,\chi)$ such that
 $P$ admits a section (i.e., is isomorphic to the trivial torsor). 
 
 We can enhance the above construction to a fibered groupoid $\Prin_{G,\XX}$ over $\sfT$. In fact,
$\Prin_{G,\XX}$ is a stack over $\sfT$. The stack $\Prin_{G,\XX}$ contains 
 $\TrivPrin_{G,\XX}$ as a full subprestack. Furthermore, since every $G$-torsor is locally trivial,
 $\Prin_{G,\XX}$ is (equivalent to) the stackification of  $\TrivPrin_{G,\XX}$.

We define a morphism of prestacks
                $$F_{pre} \: \lfloor G\backslash\XX \rfloor \to  \Prin_{G,\XX}$$ 
as follows (see \S~\ref{SS:quotient} for the definition of $\lfloor G\backslash\XX \rfloor$). 
For $T \in \sfT$, an object $x \in  \lfloor G\backslash\XX\rfloor(T)$  is, by definition,
the same as an object in $\XX(T)$. This, by Yoneda, gives a map $\f_x \: T \to \X$. 
Define $F(x)$ to be the pair $(G\times T,\chi_x)$, where $G\times T$ is viewed as a trivial $G$-torsor
over $T$, and $\chi_x:=\mu\circ (\id_G\times f_{x})$ , as in the diagram
       $$\xymatrix@M=6pt@C=12pt{ G\times T \ar[rr]^(0.47){\id_G\times f_x} && G\times \X  
         \ar[r]^(0.57){\mu} & \X.  }$$
(Note that producing $f_x$ from $x$ involves making choices, so our functor $F_{pre}$ depends on
all these choices.) Symbolically, $\chi_x$ can be written as $\chi_x \: h \mapsto h\cdot x$,
where $h$ is an element of $G$ (over $T$).
%

The effect of $F_{pre}$ on arrows is defined as follows. Given an arrow $(\gamma,g,y)$ 
                 $$\xymatrix@M=6pt{    &   y   \\   
             x   \ar[r]^{\gamma}    &  g\cdot y \ar@{..>}[u]_{g}   }$$  
in
$\lfloor G\backslash\XX\rfloor(T)$ from $x$ to $y$, we define $F_{pre}(\gamma,g,y)$ to be the pair
$(m_g,\phi)$, where $m_g \: G\times T \to G\times T$ is right multiplication by $g$ (on the first factor), and 
$\phi \: \chi_x \Ra \chi_y\circ u$ is  the composition  
   $$\xymatrix@M=6pt@C=20pt{ \chi_x \ar@{=>}[rr]^(0.47){\mu\circ (\id_G\times f_{\gamma})} 
      &&  \chi_{g\cdot y} 
         \ar@{=>}[r]^(0.43){\al_{-,g}^y} & \chi_y\circ m_g.  }$$
It is not hard to see that 
$F$ is fully faithful and it lands in  $\TrivPrin_{G,\XX}$, hence, after stackification,
we obtain an equivalence of stacks
  $$F \: [G\backslash\XX] \risom \Prin_{G,\XX}.$$

There is an alternative description of $F$ in terms of pullback torsors which is more geometric.
For any $x \in [G\backslash\XX](T)$, let $f'_x \: T \to [G\backslash\XX]$ be the morphism obtained from
Yoneda and the natural map $q: \XX\to [G\backslash\XX]$, and form the following fiber square
        $$\xymatrix@M=6pt{ P \ar[r]^{\chi} \ar[d]_{p_1}  & \XX \ar[d]^{\pq} \\
            T \ar[r]_{f'_x} &  [G\backslash\XX]}$$
Here, $P$ is the sheaf of set obtained from $T\times_{[G\backslash\XX]}\XX$
by contracting each isomorphism class to a point, as in Lemma \ref{L:representableproj}. 
The maps $p_1$ and $\chi$ are
obtained from the first and the second projection maps, respectively,
by choosing an inverse
equivalence to the projection $T\times_{[G\backslash\XX]}\XX \to P$. There is an obvious
left action of $G$ on $T\times_{[G\backslash\XX]}\XX$ in which $G$ acts on the second
factor (so the projection $\pr_2 \: T\times_{[G\backslash\XX]}\XX \to \XX$ is strictly
$G$-equivariant). This induces a $G$-action on $P$ such that $\chi$ is $G$-equivariant
(not necessarily strictly any more). 

  Sending $x$ to the pair $(P,\chi)$ gives rise to a morphism of stacks 
       $$F' \: [G\backslash\XX] \to \Prin_{G,\XX}.$$
The effect of $F'$ on arrows is defined in the obvious way. The morphism $F'$ is canonically
2-isomorphic to $F$ (hence is an equivalence of stacks). What the functor $F'$ says
is that the pair $(\XX,\id)$ is a universal pair with $\XX \to [G\backslash\XX]$ a ``$G$-torsor''
and $\id \: \XX \to \XX$ a $G$-equivariant map.

There is  a natural inverse morphism of stacks 
  $$Q \: \Prin_{G,\XX} \risom [G\backslash\XX]$$
for $F$ (or $F'$) 
which is defined as follows. Let $(P,\chi)$ be an object in $\Prin_{G,\XX}(T)$. The $G$-equivariant
map $\chi \: P \to \XX$ induced a map $[\chi] \: [G\backslash P] \to   [G\backslash\XX]$ on
the quotient stacks (\S~\ref{SS:quotient}). Since $P$ is a $G$-torsor, the natural map
$[G\backslash P] \to T$ is an equivalence of stacks. {\em Choose} an inverse $f_P \: T \to
[G\backslash P]$ for it. The composition $[\chi]\circ f_P \: T \to  [G\backslash\XX]$ determines
an object in  $ [G\backslash\XX]$ which we define to be $Q(P,\chi)$. The effect on arrows
is defined similarly (for this you do not to make additional choices).

In conclusion we have proved the following
\begin{lem}
 There are natural equivalences of stacks $[G\backslash\XX] \cong \Prin_{G,\XX}$ (induced by $F=F'$ and $Q$).
\end{lem}

\begin{rem}
  As we pointed out above, construction of the morphisms $F$, $F'$ and $Q$ 
 requires making certain choices. In the case 
  of $F$ and $F'$, the choice involves associating a map $f_x \: T \to \XX$ to an
  element $x \in \XX(T)$. The map $f_x$ is unique up to  a unique 2-morphism. In the case
  of $Q$, the choice involves choosing an inverse for the equivalence of stacks
  $[G\backslash P] \to T$. Again, such an inverse is unique up to a unique 2-morphism. The 
  conclusion is that the morphisms $F$, $F'$ and $Q$ are well defined up to a unique 2-morphism.
  As we pointed out above, $F$ and $F'$ are canonically 2-isomorphic and $Q$ is an 
  inverse equivalence to both.
\end{rem}

\subsection{Quotients of topological stacks}{\label{SS:presentation}}

In this section we assume that our Grothendieck site  $\sfT$ is either $\Top$ or $\CGTop$.
We are particularly interested in the case where the sheaf of groups $G$ indeed comes from
a topological group (denoted again by $G$). We point out that, in this case, the sheaf theoretic
notion of a $G$-torsor used in the previous subsections coincides with the usual one. More
precisely, given a topological space $T$, a sheaf theoretic $G$-torsor $P$ over $T$ always
comes from a topological space (again denoted by $P$). The reason for this is that $P$
is locally (on $T$) of the form $U\times G$, which is indeed a topological space. Gluing
these along intersections yields a topological space representing $P$.

The main result we prove in this
subsection is that if $G$ is a topological group and $\XX$ a topological stack,
then $[G\backslash \XX]$ is also a topological stack. To prove this we need two lemmas.

\begin{lem}{\label{L:representable1}}
     Let $f,g \: \XX \to \YY$ be representable morphisms of stacks. Assume further that
     the diagonal $\Delta \: \XX \to \XX\times \XX$ is representable. Then,
     $(f,g) \: \XX \to \YY \times \YY$ is representable.
\end{lem}

\begin{proof}
   We can write $(f,g)$ as a composition of two representable maps
   $\Delta \: \XX \to \XX\times \XX$ and 
   $f\times g \: \XX \times \XX \to \YY \times \YY$.
\end{proof}

\begin{lem}{\label{L:representable2}}
  Let $f \: \XX \to \YY$ be a morphism of stacks and $\YY' \to \YY$ an epimorphism
  of stacks. If the base extension $f' \: \XX' \to \YY'$ of $f$ over $\YY'$ is representable,
  then so is $f$ itself.
\end{lem}

\begin{proof}
 This is Lemma 6.3 of~\cite{Foundations}.
\end{proof}

\begin{prop}\label{P:quotientmapisrepresentable}
 Let $\XX$ be a topological stack and $G$ a topological group acting on $\XX$. Then the canonical epimorphism $\XX\to [G\backslash \XX]$ is representable.
\end{prop}
\begin{proof}
 This is Lemma~\ref{L:representable2} applied to the $2$-cartesian square $$\xymatrix@M=6pt{G\times \XX
           \ar[d]_{\mu} \ar[r]^{\pr_2}   &   
              \XX \ar[d]^{\pq}  \\   
         \XX    \ar[r]_{\pq}  
        &   [G\backslash\XX]    }$$ from \S~\ref{SS:properties} since the map $\pr_2 \: G\times \XX \to \XX$ is representable.
\end{proof}

\begin{prop}{\label{P:toplogicalquotient}}
   Let $G$ be a topological group acting on a topological stack $\XX$. Then, the
   quotient  $[G\backslash \XX]$ is also a topological stack.
\end{prop}

\begin{proof}
   We need to prove two things.
   
   \medskip
   \noindent{\em The diagonal 
   $\Delta \: [G\backslash \XX] \to [G\backslash \XX]\times [G\backslash \XX]$ 
   is representable.} To see this, we consider the 2-cartesian diagram
    $$\xymatrix@M=6pt{\XX\times_{[G\backslash \XX]}\XX \ar[r]  \ar[d]_{(\pr_1,\pr_2)} &  
      [G\backslash \XX] \ar[d]^{\Delta} \\
    \XX\times \XX \ar[r]_(0.4){\pq\times\pq} & [G\backslash \XX]\times [G\backslash \XX]}$$
  Since the map $q\times q$ is an epimorphism, it is enough to prove that the
  map $(\pr_1,\pr_2) \: \XX\times_{[G\backslash \XX]}\XX  \to  \XX\times \XX$
  is representable. As we saw in \S~\ref{SS:quotient}, this map is equivalent to
  the map $(\pr_2,\mu) \: G\times \XX \to \XX\times \XX$. The map
  $\pr_2 \: G\times \XX \to \XX$ is clearly representable. On the other hand,
   $\mu \: G\times \XX \to \XX$ is equivalent to $\pr_1$ as a map, so
  $\mu$ is  also representable. On the other hand, the diagonal $\XX \to \XX\times
  \XX$ is representable because $\XX$ is a topological stack. It follows from
  Lemma \ref{L:representable1} that  
  $(\pr_1,\pr_2) \: \XX\times_{[G\backslash \XX]}\XX  \to  \XX\times \XX$
  is  representable.
  
    \medskip
   \noindent{\em The stack $[G\backslash \XX]$ admits an atlas.} Let $X \to \XX$
   be an atlas for $\XX$. Since $q \: \XX \to [G\backslash \XX]$ is an epimorphism,
   the composition $X \to  \XX \to [G\backslash \XX]$ is an epimorphism, hence
   is an atlas for $[G\backslash \XX]$. 
\end{proof}

Using Proposition \ref{P:toplogicalquotient} we can give an explicit groupoid presentation
for $[G\backslash \XX]$ starting from a groupoid presentation
$[R \toto X]$ for $\XX$. Consider the action map $\mu \: G\times \X \to \X$.
It corresponds to a  bibundle 
        $$\xymatrix@M=6pt@C=0pt@R=10pt{ & E \ar[rd]^{\mu_2} \ar[ld]_{\mu_1} & \\
      G\times X && X}$$
The claim is that $[E \toto X]$,   
with source and target maps $s=\mu_2$ and $t=\pr_2\circ\mu_1$,
is a groupoid presentation for $[G\backslash \XX]$. It is in fact easy to see why this is the case by
staring at the 2-cartesian diagram
     $$\xymatrix@M=6pt{E \ar[r]\ar[d]_{\mu_1} \ar@/^1pc/[rr]^{\mu_2} 
         &  G\times X \ar[r]_{\pr_2} \ar[d]^{\psi} & X \ar[d]^p \\
         G\times X \ar[r]^{\id_G\times p}\ar[d]_{\pr_2} &  G\times \XX \ar[r]^{\mu}\ar[d]^{\pr_2} & 
              \XX \ar[d]^{\pq} \\
         X \ar[r]_{p} & \XX \ar[r]_{\pq} &  [G\backslash \XX] \\
                         }$$
Here $\psi \: G \times X \to G\times \XX$ is the map $(g,x) \mapsto (g,g^{-1}\cdot p(x)\big)$. Perhaps
it is helpful to remind the reader that, in general, the  bibundle $E$ associated
to a morphism of topological stacks $f \: \XX \to \YY$ given by groupoid presentations $[R\toto X]$ and
$[S\toto Y]$ is defined by the  2-cartesian diagram
 $$\xymatrix@M=6pt{E \ar[d] \ar[rr] & & Y \ar[d]^{p_Y} \\
   X \ar[r]_{p_X} & \XX \ar[r]_f & \YY }$$
                         
We now work out the composition rule in $[E \toto X]$. This relies on the analysis of the axioms 
of a group action (Definition \ref{D:action}) in
terms of bibundles. Consider the commutative square in Definition \ref{D:action}.
The composition
        $$\xymatrix@M=6pt@C=12pt{ G\times G \times \X \ar[rr]^{\id_G\times \mu} && G\times \X  
         \ar[r]^(0.57){\mu} & \X  }$$
is given by the bibundle $E\times_{s,X,t} E$ from $G\times G \times X$ to $X$ as in the diagram
  $$\xymatrix@M=6pt{  && E\times_{s,X,t} E \ar[dl]_{\pi_1}  \ar[dr]^{\pi_2} && \\
     & G\times E  \ar[dl]_{\id_G\times \mu_1}  \ar[dr]^{\id_G\times \mu_1} && E \ar[dl]_{\mu_1}  \ar[dr]^{\mu_2} & \\
  G\times G \times X \ar@{..>}[rr]_(0.47){\id_G\times \mu} && G\times X  
         \ar@{..>}[rr]_(0.57){\mu} && X
  }$$
where the maps in the cartesian square are $\pi_1(u,v)=(\pr_1\mu_1(v),u)$ and  $\pi_2(u,v)=v$. Similarly,
the composition
        $$\xymatrix@M=6pt@C=12pt{ G\times G \times \X \ar[rr]^{m\times \id_{\X}} && G\times \X  
         \ar[r]^(0.57){\mu} & \X  }$$
is given by the bibundle $B$
    $$\xymatrix@M=6pt{  && B \ar[dl]_{\pi_1'}  \ar[dr]^{\pi_2'} && \\
     & G\times G \times R  \ar[dl]_{\id_{G\times G}\times \tau}  \ar[dr]^{m\times \sigma} && E \ar[dl]_{\mu_1}  \ar[dr]^{\mu_2} & \\
  G\times G \times X \ar@{..>}[rr]_(0.47){m\times \id_{\X}} && G\times X  
         \ar@{..>}[rr]_(0.57){\mu} && X
  }$$
where $\sigma,\tau \: R \to X$ are the source and target maps of groupoid presentation of $\X$. The 2-isomorphism
$\al$ in  Definition \ref{D:action} corresponds to an isomorphism $E\times_{s,X,t} E \to B$ of bibundles. Composing this
with the projection $\pi_2' \: B \to E$ gives rise to the desired composition map $E\times_{s,X,t} E \to E$.
 
\medskip 
The above discussion immediately implies the following. 

\begin{prop}{\label{P:differentiable}}
   Let $\XX$ be a differentiable stacks and $G$ a Lie group acting smoothly on $\XX$.
   Then  $ [G\backslash \XX]$ is a differentiable stack.
\end{prop}
\begin{rem}
 The same discussion will also apply to other kind of geometric stacks, for instance for analytic or complex stacks. 
 Further, if $\XX$ is an orbifold and $G$ is finite, then $[G\backslash \XX]$ is an orbifold as well. 
\end{rem}

\section{Equivariant (co)homology of stacks}{\label{S:EquivariantCoh}}

Let $\XX$ be a topological stack with an action of a topological group $G$. We saw in
Proposition \ref{P:toplogicalquotient} that the quotient stack $[G\backslash \XX]$ is again a 
topological stack. We can apply the definitions in~\cite[\S~11,12]{Homotopytypes} to define
homotopy groups, (co)homology theories, etc., for $[G\backslash \XX]$ as recalled in 
\S~\ref{SS:classifying}. The resulting theories
are regarded as $G$-equivariant theories for $\XX$.

For example, let $H$ be the singular homology (with coefficients in any ring). Let
$(\XX,\AA)$ be a $G$-equivariant pair, namely, $\XX$ is a topological stack with 
a $G$-action, and $\AA$ is a $G$-invariant substack. 
\begin{defn} We define the {\em $G$-equivariant
singular homology} of the pair  $(\XX,\AA)$ to be
  $$H^G_*(\XX,\AA):=H_*(Y,B),$$
where $Y \to [G\backslash \XX]$ is a classifying space for $[G\backslash \XX]$
(\S~\ref{SS:classifying}), and $B\subseteq Y$ is the inverse image of $\AA$ in $Y$.
\end{defn}

More generally,  if $h$ is a (co)homology theory for topological 
spaces that is invariant under weak equivalences, we can define $G$-equivariant  
(co)homology  $h_G(\XX)$ for a $G$-equivariant stack $\XX$ (or a pair of topological
stacks) using the same procedure. 

The functoriality of the construction of the quotient stack $[G\backslash \XX]$ implies that
a $G$-equivariant morphism $f \: \XX \to \GG$ induces an natural morphism
$h(f) \:  h_G(\XX) \to h_G(\YY)$ on $G$-equivariant homology (in the covariant case)
or $h(f) \:  h_G(\YY) \to h_G(\XX)$ on  $G$-equivariant cohomology (in the contravariant case).
\begin{lem}
If $f, f' \: \XX\to\YY$ are related by a $G$-equivariant 2-morphism, the induced maps
on $G$-equivariant (co)homology are the same.
\end{lem}
\begin{proof}
 Since $f, f'$ are $G$-equivariantly 2-isomorphic, the induced maps $[f], [f']: [G\backslash \XX] \to [G\backslash \YY]$ are 
 $2$-isomorphic as well. 
\end{proof}

\subsection{(Co)homology theories that are only homotopy invariant}{\label{SS:hoparacompact}}

There are certain (co)homology theories that are only invariant under homotopy equivalences
of topological spaces. Among these are certain sheaf cohomology theories or \v{C}ech type
theories. 

As discussed in~\cite[\S~11.1]{Homotopytypes}, such (co)homology theories can be extended
to topological stacks that admit a paracompact classifying space $\varphi \: X \to \XX$ 
(satisfying the condition of Theorem \ref{T:nice}). 

\begin{prop}{\label{P:hoparacompact}}
  Let $\XX$ be a topological stack and 
  and $G$ a topological group acting on it. Let
  $[R\toto X]$ a topological groupoid presentation for $\XX$.
   Assume that  $R$, $X_0$ and $G$ are metrizable.
   Then, the quotient stack $[G\backslash \XX]$ admits a paracompact classifying space
   (which satisfies the condition of Theorem \ref{T:nice}).  
\end{prop}

Before proving the proposition we need a lemma.

\begin{lem}{\label{L:metrizable}}
  Let $[R\toto X]$ be a topological groupoid such that $R$ is metrizable. Let $T$ be a metrizable
  topological space and $f \: E \to T$ a (locally trivial) torsor for $[R\toto X]$. Then, $E$ is metrizable.
\end{lem}

\begin{proof}
  By Smirnov Metrization Theorem, we need to show that $E$ is locally metrizable, Hausdorff and
  paracompact. By local triviality of $E$ over $T$, we can find  an open cover $\{E_i\}$
  of $E$ such that each $E_i$ is homeomorphic to $T_i\times_X R_i$, where $T_i$ is an open
  subspace of $T$ trivializing $E$, and $R_i$ a subspace of $R$. 
  (The map $T_i \to X$ in the fiber product is the
  composition of the trivializing section $s_i \: T_i \to E$ with the structure map $E \to X$
  of the torsor.) It follows that $E_i=T_i\times_X R_i \subseteq T_i \times R_i$ is metrizable. Furthermore,
  since $T$ is metrizable (hence paracompact) we may assume that the open cover
  $\{T_i\}$ is locally finite. Hence, so is the open cover $\{E_i\}$ of $E$. Since each $E_i$ is metrizable
  (hence paracompact) it follows that $E$ is paracompact. Finally, to prove that $E$ is Hausdorff,
  pick two points $x$ and $y$ in $E$. If $f(x)$ and $f(y)$ are different, then we can separate them
  in $T$ by open sets $U$ and $V$, so $f^{-1}(U)$ and $f^{-1}(V)$ separate $x$ and $y$ in
  $E$. If $f(x)=f(y)$, then
  $x$ and $y$ are in some $E_i$. Since $E_i$ is Hausdorff (because it is metrizable)
  we can separate $x$ and $y$. This proves the lemma.
\end{proof}

Now we come to the proof of Proposition \ref{P:hoparacompact}.

\begin{proof}
   Consider the  explicit groupoid presentation $[E\toto X]$
  for $[G\backslash \XX]$ described in \S~\ref{SS:presentation}. 
   Recall that $E$ is a   bibundle 
        $$\xymatrix@M=6pt@C=0pt@R=10pt{ & E \ar[rd]^{\mu_2} \ar[ld]_{\mu_1} & \\
      G\times X && X}$$
  Since $E\to G\times G\times X$ is torsor for $[R\toto X]$, and $G\times X$
  is metrizable, it follows from  Lemma \ref{L:metrizable} that $E$ is metrizable. 
  The proposition follows from Proposition 8.5 of~\cite{Homotopytypes}.
\end{proof}

As a consequence, we see that   if $G$ and $\XX$ satisfy any of the conditions in
Proposition \ref{P:hoparacompact}, then any (co)homology theory $h$ that is invariant under homotopy
equivalences of topological spaces can be defined $G$-equivariantly for $\XX$. The resulting
(co)homology $h_G(\XX)$ is functorial in $\XX$ and in invariant under 2-morphisms.


\section{Group actions on mapping stacks}{\label{S:actionmapping}}

\subsection{Group actions on mapping stacks}{\label{SS:mappingaction}}

Let $\XX$ be a stack and $G$ a sheaf of groups. There is a natural  strict left $G$-action
   $$\mu \: G\times \Map(G,\XX) \to \Map(G,\XX)$$
on the mapping stack $\Map(G,\XX)$
induced from the right multiplication of $G$ on $G$. We spell out how this
works. Let $T$ be in $\sfT$ and $g \in G(T)$. We want to define the
action of $g$ on the groupoid $\Map(G,\XX)(T)$. Let $f \in \Map(G,\XX)(T)$ be an
object in this groupoid.
By definition of the mapping stack, $f$ is a map $f \: G\times T \to \XX$. We define
$g\cdot f$ by the rule $(g\cdot f)(a,t)=f(ag,t)$. More precisely, $g\cdot f$ is
the composition $f\circ m_g \: G\times T\to \XX \in \Map(G,\XX)(T)$, where
$m_g \:  G\times T \to  G\times T$ is the composition
   $$\xymatrix@M=6pt@C=48pt{  G\times T \ar[r]^(0.42){(\id_G,g)\times\id_T} & 
         (G  \times G)  \times T
         \ar[r]^(0.56){m\times \id_T} &  G\times T.}$$ 
Here, $m \: G\times G \to G$ is the multiplication in $G$.
The action of $g$ on arrows of $\Map(G,\XX)(T)$ is defined similarly.

Given a map $\XX \to \YY$ of stacks, the induced map $\Map(G,\XX) \to \Map(G,\YY)$ 
is strictly $G$-equivariant.

The case we are interested in is where $G=S^1$ is the circle. We find that the loop stack
$\LXX$ has a natural strict $S^1$-action.

\subsection{Interpretation of   $[G\backslash\!\Map(G,\XX)]$ in terms of 
torsors}{\label{SS:quotientviatorsor}}

We saw in \S~\ref{SS:mappingaction} that for every stack $\XX$ and every sheaf of groups $G$,
the mapping stack $\Map(G,\XX)$ has a natural left $G$-action. Our goal is to understand
the quotient stack $[G\backslash\!\Map(G,\XX)]$ of this action in the spirit of 
\S~\ref{SS:torsors}.

Let $\XX$ be a stack with an action of a sheaf of groups $G$. Let $T$ be an object in $\sfT$.
We define the groupoid $\Prinu_{G,\XX}(T)$ as follows. 
 {\small   $$\ob\Prinu_{G,\XX}(T)=\left\{\begin{array}{rcl}
              (P,\chi) & | &  P\to T \ \ \text{left $G$-torsor}, \\
                    &   &    \chi \: P \to \XX \ \
              \text{morphism of stacks}   
                                                          \end{array}\right\}$$
      $$\Mor_{\Prinu_{G,\XX}(T)}\big((P,\chi),(P',\chi')\big)=\left\{\begin{array}{rcl}
              (u,\phi) & | &  u \: P \to P'  \ \text{map of $G$-torsors}, \\
                    &   &    \phi \: \chi \Ra \chi'\circ u \ \ \text{2-morphism}
                                                          \end{array}\right\}$$  }     
 We can enhance the above construction to a fibered groupoid $\Prinu_{G,\XX}$ over $\sfT$. In fact,
$\Prinu_{G,\XX}$ is a stack over $\sfT$. 

The definition of the stack $\Prinu_{G,\XX}$ is very similar to $\Prin_{G,\XX}$, except that we have
dropped the $G$-equivariance condition on $\chi$ and $\phi$. A $T$-point of $\Prinu_{G,\XX}$
should be regarded as a `family of $G$-torsors in $\XX$ parametrized by $T$'. In the case
when $G=S^1$, $\Prinu_{S^1,\XX}$ is the stack of unparametrized loops in $\XX$.

There is a natural morphism of stacks  
                    $$p \: \Map(G,\XX) \to \Prinu_{G,\XX}$$ 
which sends $f  \in 
\Map(G,\XX)$, $f \: G\times T \to \XX$, to the pair $(G\times T,f)$ in $\Prinu_{G,\XX}(T)$.
Here, we are viewing $G\times T$ as a trivial $G$-torsor over $T$.

\begin{prop}{\label{P:unparametrized}}
    There is a canonical (up to a unique 2-morphism) equivalence of stacks
      $$\Phi \: [G\backslash\!\Map(G,\XX)]\risom\Prinu_{G,\XX}$$
    making the diagram
     $$\xymatrix@M=6pt@C=-14pt@R=10pt{ & \Map(G,\XX) \ar[rd]^p \ar[ld]_{\pq} & \\
       [G\backslash\!\Map(G,\XX)]\ar[rr]^(0.55){\sim}_(0.55){\Phi} && \Prinu_{G,\XX} }$$
     canonically 2-commutative.   
\end{prop}

\begin{proof}
  We define the effect of $\Phi$ on objects as follows.
  Let $T$ be an object in $\sfT$. As we saw in \S~\ref{SS:torsors}, a map 
$T \to  [G\backslash\!\Map(G,\XX)]$ is characterized by a pair $(P,\chi)$, where $P \to T$
is a left $G$-torsor and $\chi \: P \to  [G\backslash\!\Map(G,\XX)]$ is a $G$-equivariant
map. Unravelling the definition of $G$-equivariance, we find the following
description of $T$-points of $[G\backslash\!\Map(G,\XX)]$. A $T$-point of 
$[G\backslash\!\Map(G,\XX)]$ is given by 
 triple $(P,\chi,\sigma)$, where $P \to T$
is a left $G$-torsor,  $\chi \:  G \times P \to \XX$ is a morphism, and
$\sigma$ is a 2-morphism 
as in the diagram
         $$\xymatrix@R=22pt@C=36pt@M=5pt{  G \times G \times P
               \ars{[r]}{-1ex}{0.4}  ^(0.55){\id_G\times \mu}  
               \art{[d]}{1ex}{0.4}   _{m\times\id_P} 
                                  &  G\times P   \ar[d]^{\chi}     \ar  @/^/ @{=>}^{\sigma} "s";"t" \\
                                   G \times  P   \ar[r]_{\chi}       &   \XX        } $$  
where $\mu \: G \times P \to P$ is the  action of $G$ on $P$ and $m \: G \times G \to G$
is multiplication in $G$.
The following equality is required to be satisfied:
  \begin{itemize}
      \item $\sigma_g^{a,hp}\sigma_h^{ag,p}=\sigma_{gh}^{a,p}$, 
         for every $a,g,h$ in $G$ and $p$ in $P$.
  \end{itemize} 
 Here, $\sigma_g^{a,p}$ stands for the arrow $\sigma(a,g,p) \:  \chi(a,gp) \to \chi(ag,p)$.
 
 Observe that, because of the $G$-equivariance condition above,  
 $\chi\:  G \times P \to \XX$ is uniquely (up to a unique 2-morphism)
 determined by its restriction to $\{1\}\times P$, that is, by the composition
     $$\xymatrix@M=6pt@C=30pt{  \chi_1 \: P \ar[r]^{(1_G,\id_P)} & G\times P \ar[r]^{\chi}
      & \XX.}$$ 
 We define $\Phi(P,\chi,\sigma):=(P,\chi_1)$.    
    
Given two   $T$-point $(P,\chi,\sigma)$ and $(P',\chi',\sigma')$ in $[G\backslash\!\Map(G,\XX)]$, 
a morphism between them is a pair $(u,\phi)$, where 
$u \: P \to P' $ is a map of $G$-torsors and 
$\phi \: \chi \Ra \chi'\circ u$ is a 2-morphism satisfying 
     \begin{itemize}
      \item $\sigma_g^{a,p}\phi^{ag,p}=\phi^{a,gp}\sigma'_g{}^{a,u(p)}$, 
         for every $a,g$ in $G$ and $p$ in $P$.
  \end{itemize} 
  Here, $\phi^{a,p}$ stands for  the arrow $\phi(a,p) \: \chi(a,p) \to \chi'(a,u(p)\big)$. As in
  the case of $T$-points, the $G$-equivariance implies that
  an  $\phi$  is uniquely   determined by its restriction $\phi_1$ to $\{1\}\times P$, which
  is obtained by precomposing $\phi$ by 
 $(1_G,\id_P) \: P \to  G\times P$,
      $$\xymatrix@M=6pt@C=30pt{  P \ar[r]^(0.4){(1_G,\id_P)} & G\times P
      \ar@{}@<1.5ex> [r] | (0.55){}="s" \ar@<2ex>[r]^{\chi}
      \ar@{}@<-1.5ex>[r] | (0.55){.}="t" \ar@<-2ex>[r]_{\chi'\circ u} 
      \ar   @{=>}_{\phi} "s";"t"  
      & \XX.}$$ 
 We define $\Phi(u,\phi):=(u,\phi_1)$.
 
 We leave it to the reader to verify the last part of the proposition  (2-commutativity of the triangle).
\end{proof}

\begin{ex}{\label{E:unparametrized}}
   Suppose that $\X=[H\backslash X]$ is the quotient stack of the action of a topological group $H$ on
   a topological space $X$. Then, a $T$-point of $ [G\backslash\!\Map(G,\XX)]$ is
   a sequence
     $$Q\to  P\to T$$
   together with a continuous map $\chi \: Q \to X$, 
    where  $Q$ is an $H$-torsor over $P$  
    and $P$ is a $G$-torsor over $T$. The map $\chi$ is assumed to be
    $H$-equivariant. A morphism between such $T$-points is a
    commutative diagram
      $$\xymatrix@M=6pt@R=0pt{ Q \ar[dd]_{u_2} \ar[r] & P  \ar[dd]_{u_1} \ar[rd] &  \\
           &  &   T  \\
          Q'  \ar[r] & P'  \ar[ru] & }$$
     such that $u_1$ is $G$-equivariant, $u_2$ is $H$-equivariant, and $\chi=\chi'\circ u_2$.     
\end{ex}

\subsection{A slight generalization}{\label{SS:generalization}}

The set up being as in \S~\ref{SS:mappingaction}, let $H$ be 
another sheaf of groups, and $F\: G \to H$ a homomorphism. 
We have an induced morphism of mapping stacks
    $$\Map(H,\XX) \to \Map(G,\XX).$$
The homomorphism $F$ also gives rise to an action of $G$ on $\Map(H,\XX)$ 
making the above map a $G$-equivariant map.
Therefore, we have an induced map on the quotient stacks
  $$F^* \: [G\backslash\!\Map(H,\XX)] \to [G\backslash\!\Map(G,\XX)].$$

There is a torsor description  for $ [G\backslash\!\Map(H,\XX)] $ and $F^*$ as follows.
Let $\XX$ be a stack with an action of a sheaf of groups $G$. Let $T$ be an object in $\sfT$.
We define the stack $\Prinu_{G\to H,\XX}$ by the following rule: 
 {\small   $$\ob\Prinu_{G\to H,\XX}(T)=\left\{\begin{array}{rcl}
              (P,\chi) & | &  P\to T \ \ \text{left $G$-torsor}, \\
                    &   &    \chi \:  P_H \to \XX \ \
              \text{morphism of stacks}   
                                                          \end{array}\right\}$$
      $$\Mor_{\Prinu_{G\to H,\XX}(T)}\big((P,\chi),(P',\chi')\big)=\left\{\begin{array}{rcl}
              (u,\phi) & | &  u \: P \to P'  \ \text{map of $G$-torsors}, \\
                    &   &    \phi \:  {\chi} \Ra \chi'\circ u_H \ \ \text{2-morphism}
                                                          \end{array}\right\}$$  }     
Here, $P_H:=H\overset{G}{\times} P$ stands for extension of structure group from $G$ to $H$, and $u_H$ is the 
induced map on the extensions. As in Proposition \ref{P:unparametrized}, it can be shown that
there is a canonical (up to a unique 2-morphism) equivalence of stacks
      $$[G\backslash\!\Map(H,\XX)]\risom\Prinu_{G\to H,\XX}.$$
The torsor description of the map $F^* \: [G\backslash\!\Map(H,\XX)] \to [G\backslash\!\Map(G,\XX)]$
is given by
     $$ \Prinu_{G\to H,\XX} \to \Prinu_{G,\XX},$$
     $$(P,\chi) \mapsto (P,\chi\circ f),$$
where $f \: P \to P_H$ is the maps induced from the map $P \to H\times P$, $x \mapsto (1,x)$.

In the following examples  we consider two extreme case of the above construction.
\begin{ex}
 In the case where $H$ is the trivial group, denoted $1$,
  we have 
     $$[G\backslash\!\Map(1,\XX)]\cong\Prinu_{G\to 1,\XX}\cong BG\times \X\cong [G\backslash\!\X],$$
  where in the last term $G$ acts trivially on $\X$.    
   The map $F^*$ coincides with the map
       $$[G\backslash\!c] \:  [G\backslash\!\X] \to [G\backslash\!\Map(G,\XX)],$$
   where $c \: \X \to \Map(G,\XX)$ is the morphism parametrizing constant maps $G \to \XX$.
\end{ex}

\begin{ex}
 In the case where $G=1$ is the trivial group,
 we have
   $$[1\backslash\!\Map(H,\XX)]\cong \Map(H,\XX)\cong\Prinu_{1\to H,\XX},$$ 
  and the map $F^* \: \Map(H,\XX) \to \XX$ is the evaluation map at the identity element of $H$.
\end{ex} 
  
\subsection{Existence of a groupoid presentation for $[G\backslash\!\Map(H,\XX)]$}{\label{SS:atlas}}
   
\begin{prop}{\label{P:mappingquotient}}
  Let $G$ and $H$ be topological groups and  $G \to H$ a homomorphism. Assume that $H$ is
  compact.  Let $\XX$ be a topological stack. Then,
  $[G\backslash\!\Map(H,\XX)]$ is a topological stack.
\end{prop}   

\begin{proof}
This follows from Theorem \ref{T:mapping} and Proposition  \ref{P:toplogicalquotient}.
\end{proof}

The above proposition allows us to define $G$-equivariant (co)homology of the mapping
stack $\Map(H,\XX)$ using \S~\ref{S:EquivariantCoh}. 

Although, strictly speaking,
Proposition \ref{P:mappingquotient} may not be true when $H$ is not compact, there is a
slightly weaker version of it which is sufficient for the purpose of defining the $G$-equivariant
(co)homology. We will not need this result in this paper, but we state it as we think it may
be useful in other applications.

\begin{prop}{\label{P:mappingquotient2}}
  Let $H$ be a locally compact topological group and $\XX$ a topological stack. Then,
  there exists a topological stack $\YY$ and
  a morphism of stacks $f \: \YY \to [G\backslash\!\Map(H,\XX)]$ with the property
  that, for every paracompact topological space $T$, the induced
  map  $f(T) \: \YY(T) \to  [G\backslash\!\Map(H,\XX)](T)$ on $T$-points is 
  an equivalence of groupoids.
\end{prop}  

\begin{proof} 
   By~\cite{Mapping}  Theorem 4.4,  $\Map(H,\XX)$
   is a paratopological stack (see~\cite[\S~2.2]{Mapping}  for definition).  The  
   proof of Proposition  \ref{P:toplogicalquotient} can be repeated here to show 
   that  $ [G\backslash\!\Map(H,\XX)]$ is paratopological.   
   The claim  now follows from~\cite[Lemma 2.4]{Mapping}.   
\end{proof}

The above proposition says that, although $[G\backslash\!\Map(H,\XX)]$ may not be a
topological stack, from the eye of paracompact topological spaces $T$ it behaves like one.
In particular, since most homotopical invariants (such as, homotopy groups, (co)homology, etc.)
are defined using paracompact spaces (spheres, simplices, etc.), they make sense
for  $[G\backslash\!\Map(H,\XX)]$.

%
%

\subsection{$S^1$-action on the loop stack}{\label{SS:S1actionloop}}

Let $\XX$ be a topological stack. Then the loop stack $\LXX:=\Map(S^1,\XX)$ is again
a topological stack (Theorem \ref{T:mapping}). By \S~\ref{SS:mappingaction}, there is
a strict $S^1$-action on $\LXX$. The quotient stack $[S^1\backslash\LXX]$  of this action
is again a topological stack (Proposition \ref{P:mappingquotient}). The stack $[S^1\backslash\LXX]$
parametrizes unparametrized loops in $\XX$, in the sense that, for
every topological space $T$, the groupoid  $[S^1\backslash\LXX](T)$ of  its $T$-points
is naturally equivalent to the groupoid
     $$\ob\Prinu_{S^1,\XX}(T)=\left\{\begin{array}{rcl}
              (P,\chi) & | &  P\to T \ \ \text{left $S^1$-torsor}, \\
                    &   &    \chi \: P \to \XX \ \
              \text{morphism of stacks}   
                                                          \end{array}\right\}$$
      $$\Mor_{\Prinu_{S^1,\XX}(T)}\big((P,\chi),(P',\chi')\big)=\left\{\begin{array}{rcl}
              (u,\phi) & | &  u \: P \to P'  \ \text{map of $S^1$-torsors}, \\
                    &   &    \phi \: \chi \Ra \chi'\circ u \ \ \text{2-morphism}
                                                          \end{array}\right\}$$    
                                                          
In the case where $\X=[H\backslash X]$, Example \ref{E:unparametrized}
(with $G=S^1$) gives a more explicit description of the groupoid of $T$-points of 
$[S^1\backslash\LXX]$.

\section{Homotopy type of $[G\backslash \Map(G,\XX)]$}{\label{S:Homotopytype}}

\subsection{Classifying space of $\Map(G,\XX)$ }\label{SS:ClassSpaceMap}

\begin{lem}{\label{L:univwe}}
  Let $f \: \XX \to \YY$ be a representable $G$-equivariant morphism of topological stacks.
  Then, the induced
  map $[f] \: [G\backslash\XX] \to [G\backslash\YY]$ is also representable. If $f$ is 
  a universal weak equivalence (\S~\ref{SS:classifying}), then so is $[f]$.
\end{lem}

\begin{proof}
  The first statement follows from the 2-cartesian diagram (\S~\ref{SS:quotient})
        $$\xymatrix@M=6pt{\XX
           \ar[d]_{\pq_{\XX}} \ar[r]^{f}   &   
              \YY \ar[d]^{\pq_{\YY}}  \\   
         [G\backslash\XX]    \ar[r]_{[f]}  
        &    [G\backslash\YY]  }$$
   and Lemma 6.3 of~\cite{Foundations}.     To prove the second part, let $T \to  [G\backslash\YY] $
   be a map with $T$ a topological space. Let $P:=T\times_{[G\backslash\YY] }\YY$ be the
   corresponding $G$-torsor on $T$, with $\chi \: P \to \YY$ the second projection map
   (see \S~\ref{SS:torsors}).  We need to show that
   the projection map
     $$F \: T\times_{ [G\backslash\YY]} [G\backslash\XX] \to T$$
    is a weak equivalence. We have
       $$T\times_{ [G\backslash\YY]} [G\backslash\XX]\cong G\backslash(P\times_{\YY} \XX),$$
   where the $G$-action on the right hand side is induced by the one on $P$. 
   (Note that the action of $G$ on   $P\times_{\YY} \XX$ is free.)   Using the above isomorphism, the
   map $F$ is the same as the map
        $$G\backslash(P\times_{\YY} \XX) \to  G\backslash P,$$
   induced from the projection
      $P\times_{\YY} \XX \to  P$
    after passing to the quotient of the free $G$-actions.  
    (here, we have written $T$ as $G\backslash P$ for clarity).    The latter is a weak equivalence by
    assumption, therefore so is the one after passing to the (free) $G$-quotients.
\end{proof}

\begin{prop}{\label{P:classifyingquotient}}
  Let $\XX$ be a topological stack and $\varphi \: X \to \XX$ a classifying
  space for it as in Theorem \ref{T:nice}.
  Let $G$ be a paracompact topological group. Then, there is a natural
  map 
     $$\Map(G,X)\times_G EG \to [G\backslash \Map(G,\XX)]$$
  making the Borel construction $\Map(G,X)\times_G EG$ a classifying space for
  $[G\backslash \Map(G,\XX)]$.
\end{prop}


\begin{proof}
 By Theorem \ref{T:invariance},
 the map $\Map(G,X) \to \Map(G,\XX)$  is a universal weak equivalence and makes $\Map(G,X)$
 a classifying space for $\Map(G,\XX)$. 
 By Lemma \ref{L:univwe}, the induced morphism
 $[G\backslash \Map(G,X)] \to [G\backslash \Map(G,\XX)]$  
 is representable and a universal weak equivalence. We also know that there is a natural
 map $\Map(G,X)\times_G EG \to [G\backslash \Map(G,X)]$
 making the Borel construction $\Map(G,X)\times_G EG$ a classifying space for
 $[G\backslash \Map(G,X)]$. Composing these two maps give us the desired
 universal weak equivalence  $\Map(G,X)\times_G EG \to [G\backslash \Map(G,\XX)]$.
\end{proof}

\subsection{Homotopy type of $[S^1\backslash \LXX]$}

Specifying the results of Proposition~\ref{P:classifyingquotient} (in the last section) to $G=S^1$ we obtain:
\begin{cor}{\label{C:classifyingquotient}}
  Let $\XX$ be a topological stack and $\varphi \: X \to \XX$ a classifying
  space for it as in Theorem \ref{T:nice}.
   Then, there is a natural map 
     $$LX\times_{S^1} \mathbb{CP}^{\infty} \to [S^1\backslash \LXX]$$
  making the Borel construction $LX\times_{S^1} \mathbb{CP}^{\infty}$ a classifying space for
  $[S^1\backslash \LXX]$.  
\end{cor}

\begin{cor}{\label{C:specseq}}
 Let $\XX$ be a topological stack. There is a natural spectral sequence $E^1_{*,*}$ converging to $H_*^{S^1}(\LXX)$ whose first page $E^1_{p,q}$  is isomorphic to
 $E^1_{p,q} \cong H_{p-q}(\LXX)$ with differential $d^1: E^1_{p,q}\to E^1_{p,q-1}$ given by the $S^1$-action operator $D:H_{p-q}(\LXX)\to H_{p-(q-1)}(\LXX)$ defined below~\eqref{eq:DefDelta}.
\end{cor}
\begin{proof} Let $\varphi \: X \to \XX$ be a classifying
  space for $\XX$ as in Theorem \ref{T:nice}.
 By Corollary~\ref{C:invariance} and Corollary~\ref{C:classifyingquotient}, we are left to the same question 
 with $\LXX$ replaced by $LX$. The spectral sequence is now the usual spectral sequence computing 
 $S^1$-equivariant homology of a $S^1$-space.
\end{proof}

\begin{rmk}
Usually by strings on a manifold $M$, one means free loops on $M$, up to reparametrization by  (orientation preserving) homeomorphism (or diffeomorphism). 
Similarly to \S~\ref{SS:mappingaction}, if $\XX$ is a topological stack, the group 
$\mathrm{Homeo}^{+}(S^1)$ of orientation preserving homeomorphism of the circle acts 
in an natural way on the free loop stack $\LXX$ (through its natural action on $S^1$) and we call the quotient stack $[\mathrm{Homeo}^{+}(S^1)\backslash \LXX]$, the \emph{stack  of strings} of $\XX$.  
The canonical inclusion $S^1\hookrightarrow  \mathrm{Homeo}^{+}(S^1)$ induces a map of stacks $[S^1\backslash \LXX]\to [\mathrm{Homeo}^{+}(S^1)\backslash \LXX]$ which is a (weak) homotopy equivalence by Proposition~\ref{P:StringisS1} below. This justifies the terminology of \emph{string bracket} in Corollary~\ref{C:LieLoopStack} and that we only consider $[S^1\backslash \LXX]$ in Sections~\ref{S:Stringproduct}, \ref{S:Lemma} and \ref{S:Goldman}.
 \end{rmk}
 
\begin{prop}\label{P:StringisS1} 
Let $\YY$ be a $\mathrm{Homeo}^{+}(S^1)$-stack. The canonical map 
$[S^1\backslash \YY]\to [\mathrm{Homeo}^{+}(S^1)\backslash \YY]$ is a 
weak homotopy equivalence and, in particular, induces equivalence in (co)homology.
\end{prop}

\begin{proof} 
If $\YY$ is a $\mathrm{Homeo}^{+}(S^1)$-stack, then it is both a 
$\mathrm{Homeo}^{+}(S^1)$-torsor over $[\mathrm{Homeo}^{+}(S^1)\backslash \YY]$ and a 
$S^1$-torsor over $[S^1\backslash \YY]$. 
 Since the canonical map $S^1\hookrightarrow  \mathrm{Homeo}^{+}(S^1)$ is a 
 homotopy equivalence, the result follows from the  homotopy long exact 
 sequence~\cite[Theorem 5.2]{Fibrations}.
\end{proof}

\begin{rmk} 
The same proof applies to prove that 
 if $\YY$ is a differentiable stack endowed with an action of the group 
 $\textrm{Diff}^+(S^1)$ (of orientation preserving diffeomorphism of the circle),  then
 the quotient map $[S^1\backslash \YY]\to [\mathrm{Diff}^{+}(S^1)\backslash \YY]$ is a weak homotopy equivalence.
\end{rmk}

\section{Transfer map  and the Gysin sequence for $S^1$-stacks}{\label{S:Transfer}}

\subsection{Transfer map for $G$-stacks}{\label{SS:transfer}}
We  define natural transfer homomorphisms in (co)homology associated to the projection 
$q:\YY\to [S^1\backslash \YY]$ of  an $S^1$-stack (our main case of interest) and more generally for 
$\YY\to [G\backslash \YY]$ when $G$ is a Lie group.
We will use the framework for transfer, i.e., Gysin, maps introduced in~\cite{BGNX} and briefly 
recalled in~\S~\ref{SS:bivarianttheory}.

\begin{lem}\label{L:orientationdottoS1} 
Let $G$ be a compact Lie group. There is a strong orientation class \cite[\S~8]{BGNX} 
$$ \theta_{G} \in H^{-\dim(G)}\big(*\to [G\backslash *]\big).$$
In particular, there is a strong orientation class 
$$ \theta_{S^1} \in H^{-1}(*\to [S^1\backslash *])\cong {\bf k}.$$
\end{lem}

\begin{proof}
 The canonical map $*\to [G\backslash *]$ factors as $*\cong [G\backslash G]\to [G\backslash *]$. 
 Hence the existence of the class $\theta_{G}$ follows from~\cite[Proposition 8.32]{BGNX}. In the special case of $S^1$, 
 this class can be computed easily  from the factorization
 $[S^1\backslash S^1] \hookrightarrow [S^1\backslash \mathbb{R}^2] \to [S^1\backslash *]$ where the first map 
 is the canonical inclusion of $S^1$ as the unit sphere of $\mathbb{R}^2$ and the last map is a bundle map. Indeed,
 by definition of bivariant classes~\cite{BGNX}, we have that $H^i(*\to [S^1\backslash *])$ is isomorphic to 
 $H^{i+2}_{S^1}(\mathbb{R}^2, \mathbb{R}^2\setminus S^1)$ and the isomorphism
 $H^{-1}(*\to [S^1\backslash *])\cong {\bf k}$ now follows from the long exact sequence of a pair 
 (in $S^1$-equivariant cohomology). 
\end{proof}

Let $G$ be a compact Lie group and $\YY$  a $G$-stack. Since the canonical map 
$\YY\to *$ is $G$-equivariant, we know from Section~\ref{SS:quotient} that the diagram
\begin{equation}\label{eq:pullbackS1stack}
\xymatrix@M=6pt{\YY
           \ar[d] \ar[r]^{q}   &   
               [S^1\backslash\YY] \ar[d]^{u}  \\   
        { *}   \ar[r]_{q}  &   [S^1\backslash *]  }
\end{equation}
is 2-cartesian.  Thus, Lemma~\ref{L:orientationdottoS1} and~\cite[\S~9.1, 9.2]{BGNX} 
provide us with canonical Gysin maps as in the following definition.

\begin{defn}\label{D:transfer} 
Let $\YY$ be a $G$-stack, with $G$ a compact Lie group.
The homology \emph{transfer} map $T^G:H_*^{G}(\YY)\to H_{*+\dim(G)}(\YY)$ associated to $\YY$ is the Gysin map
$$T^G:=\theta_{G}^!= x\mapsto u^*\big (\theta_{G})\big)\cdot x, \quad \mbox{for } x\in
 H_*(\YY)=H^{-*}(\YY\to pt).  $$
The cohomology transfer map $T_G:H^*(\YY)\to H^{*-\dim(G)}_G(\YY)$ is similarly defined to be the Gysin map
$$T_G:={\theta_{G}}_!= x\mapsto  (-1)^{i} q_*\big(x\cdot u^*(\theta_{G})\big), \quad \mbox{for } x\in
 H^i(\YY)=H^{i}(\YY\stackrel{id}\to \YY).  $$
If $G=S^1$, we  denote  the transfer map $T^{S^1}$ simply by
 $T:H_*^{S^1}(\YY)\to H_{*+1}(\YY)$   and call it the \emph{transgression} map. In other words, 
$$T = \theta_{S^1}^!:= x\mapsto u^*\big (\theta_{S^1})\big)\cdot x, \quad \mbox{for } x\in
 H_*(\YY)=H^{-*}(\YY\to pt).  $$
\end{defn}

\begin{prop}\label{P:transferisnatural}
The transfer map is natural, that is, if $f:\ZZ\to\YY$ is a 
$G$-equivariant morphism  of topological stacks, then 
$$f_*\circ T^G = T^G\circ [G\backslash f]_* : H_{*}^{G}(\ZZ) \to H_{*+\dim(G)}(\YY).$$
Similarly,
 $$[G\backslash f]^*\circ T_G = T_G\circ f^* : H^{*}_{G}(\YY) \to H^{*-\dim(G)}(\ZZ).$$ 
\end{prop}

\begin{proof}
 This is an easy application of the naturality properties of Gysin maps~\cite[\S~9.2]{BGNX} applied to the 
cartesian square~\eqref{eq:quotientequivcartesian} associated to a $G$-equivariant
 morphism of topological stacks.
\end{proof}

\subsection{Gysin sequence for $S^1$-stacks }
\label{SS:GysinSequence}

We now establish the Gysin sequence associated to an $S^1$-stack $\YY$.

\begin{prop}\label{P:GysinSequence}
Let $\YY$ be an $S^1$-stack. There is a (natural with respect to $S^1$-equivariant maps of stacks) 
long exact sequence in homology
     $$\dots \to H_i(\YY) \stackrel{q_*}\to H_i^{S^1}(\YY) \stackrel{\cap c}\to H_{i-2}^{S^1}(\YY) 
                   \stackrel{T}\to H_{i-1}(\YY) \stackrel{q_*}\to H_{i-1}^{S^1}(\YY)  \to\dots,  $$
where $q: \YY \to [S^1\backslash \YY]$ is the quotient map,  $T$ is the transgression map (Definition~\ref{D:transfer}),
and $c$ is the fundamental 
class of the $S^1$-bundle $\YY \to [S^1\backslash \YY]$ (that is, the Euler class of the associated oriented disk bundle over $[S^1\backslash \YY]$).
\end{prop}

\begin{proof}
The map  $q: \YY \to [S^1\backslash \YY]$ makes $\YY$ into 
an $S^1$-torsor over $[S^1\backslash \YY]$ (by \S~\ref{SS:torsors}). This map is representable by
 Proposition~\ref{P:quotientmapisrepresentable}.

 Let $Z\to [S^1\backslash \YY] $ be a classifying space for $[S^1\backslash \YY] $ as in Theorem~\ref{T:nice},
  and let $Y\to \YY$ be the classifying space of $\YY$ obtained by pullback along $q$.   
  Then, $Y\to Z$ is a principal $S^1$-bundle (Lemma~\ref{L:representableproj}) and  the long exact sequence 
  in the proposition is the Gysin sequence
  $$\dots \to H_i(Y) \to H_i(Z) \stackrel{\cap c}\to H_{i-2}(Z) 
                   \stackrel{\tilde{T}}\to H_{i-1}(Y) \to H_{i-1}(Z)  \to\dots  $$
 of this $S^1$-principal bundle under the isomorphisms $H_i(Z)\cong H_i([S^1\backslash \YY])\cong H^{S^1}_i(\YY)$ and $H_j(Y)\cong H_j(\YY)$. Here, $c$ is the Euler class of the associated disk bundle (that is, the mapping cylinder of $Y\to Z$).
 
 By~\cite[\S~9]{BGNX}, the definition of cup product by bivariant classes~\cite[\S~7.4]{BGNX}, and the discussion in Lemma~\ref{L:orientationdottoS1}, the transgression map of Definition~\ref{D:transfer} is induced (under the above isomorphisms) by the connecting homomorphism in the long exact sequence of the pair $(E, Y)$ where $E$ is the disk bundle associated to the $S^1$-principal bundle $Y\to Z$. Hence,  the transgression map $T$ is identified with the map $ H_{i-1}(Z)\stackrel{\tilde{T}}\to H_{i}(Y)$ in the Gysin sequence.
\end{proof}

Let $D: H_i(\YY)\to H_{i+1}(\YY)$ be the operator, called \emph{the $S^1$-action operator}, 
defined as the composition
\begin{equation}\label{eq:DefDelta} 
     D\: H_i(\YY) \stackrel{ [S^1]\times-}\longrightarrow H_{i+1}(S^1\times \YY) \stackrel{\mu_*}\longrightarrow H_{i+1}(\YY),
\end{equation}
where $[S^1]\in H_1(S^1)$ is the fundamental class and $\mu: S^1\times \YY \to \YY$ is the $S^1$-action.

\begin{lem}\label{L:DgivenbyGysin}
  The operator $D$ is equal to the composition
  $$ H_i(\YY) \stackrel{q_*}\to H_i^{S^1}(\YY) \stackrel{T}\to H_{i+1}(\YY) . $$
In particular, $D\circ D=0$ and $D$ is natural with respect to $S^1$-equivariant maps of stacks.
\end{lem}
\begin{proof}
 By Proposition~\ref{P:transferisnatural} and~\ref{P:GysinSequence}, we only need to prove that $D=T\circ q_*$. From \S~\ref{SS:properties}, we get a tower of 2-cartesian diagrams 
 $$\xymatrix@M=6pt{S^1\times \YY
           \ar[d]_{\pr_2} \ar[r]^{\mu}   &   
              \YY \ar[d]^{\pq}  \\   
         \YY    \ar[r]_{\pq}  \ar[d]
        &   [S^1\backslash\YY]  \ar[d] \\
        {*} \ar[r] & [S^1\backslash *] }$$
        The result follows from naturality of Gysin maps, see~\cite[\S~9.2]{BGNX}.
\end{proof}

\section{Equivariant String Topology for free loop stacks}{\label{S:Stringproduct}}
In this section we look at natural algebraic operations on strings of a stack $\XX$, that is, on the quotient stack 
$[S^1\backslash \LXX]$.

\subsection{Batalin-Vilkovisky algebras}
We first  quickly recall the definition of a \BV-algebra and its  underlying Gerstenhaber algebra structure.

A {\bf Batalin-Vilkovisky algebra} ({\bf \BV-algebra} for short) is a graded commutative associative
algebra with a degree $1$ operator $D$  such that $D^2=0$ and the  following identity
is satisfied:
     \begin{multline} \label{eq:BVidentity} D(abc)-D(ab)c-(-1)^{|a|}aD(bc)-(-1)^{(|a|+1)|b|} bD(ac)+
      \\ +D(a)bc+(-1)^{|a|}aD(b)c+
      (-1)^{|a|+|b|}abD(c)=0.\end{multline}
In other words, $D$ is a second-order differential operator. Note that we do \emph{not} assume \BV-algebras to be unital. 

 Let $(A, \cdot, D)$ be a \BV-algebra. We can define a degree 1 binary operator $\{-;-\}$ by the following formula:
\begin{eqnarray}\label{eq:Gerstenhaberbracket}
 \{a;b\} &=& (-1)^{|a|}D(a\cdot b) -(-1)^{|a|}D(a)\cdot b -a \cdot D(b).
\end{eqnarray}
The BV-identity~\eqref{eq:BVidentity} and commutativity of the product imply that $\{\,;\,\}$ is a 
derivation in each variable (and anti-symmetric with respect to the degree shifted down by 1). 
Further, the relation $D^2=0$ implies the (graded) Jacobi identity for $\{\,;\,\}$. In other words, 
$(A,\cdot, \{-;-\})$ is a \emph{Gerstenhaber} algebra, that is a commutative graded algebra 
equipped with a bracket $\{-;-\}$ that makes $A[-1]$ into a graded Lie algebra satisfying a graded 
Leibniz rule~\cite{Ge}.

It is standard (see~\cite{Get}) that a graded commutative algebra $(A,\cdot)$ 
equipped with a degree 1 operator $D$ such that $D^2=0$ is a \BV-algebra 
if and only if the operator $\{-;-\}$ defined by the formula~\eqref{eq:Gerstenhaberbracket} is a 
derivation of the second variable, that is
\begin{eqnarray}\label{eq:BVidentity2}
 \{a;bc\}= \{a;b\}\cdot c + (-1)^{|b|(|a|+1)}b\cdot\{a;c\}.
\end{eqnarray}

\smallskip

The following lemma was first noticed by Chas-Sullivan~\cite{CS}.
\begin{lem}\label{L:BVtoLieviaGysin}
Let $(B_*, \star, \Delta)$ be a  BV-algebra and $H_*$  a graded module related to it
by an \lq\lq{}$S^1$-Gysin exact sequence,\rq\rq{} that is, sitting in a long exact sequence    
$$\dots \to B_i \stackrel{q}\to H_i\stackrel{c}\to H_{i-2}\stackrel{T}\to B_{i-1}\stackrel{q}\to H_{i-1}  \to\dots  $$
such that $\Delta= T\circ q$. Then, we have the following:

\begin{enumerate}
  \item The composition 
      \begin{equation*} \{-,-\} \: H_{i-2}\otimes H_{j-2}\stackrel{T\otimes T}\longrightarrow  B_{i-1}\otimes 
       B_{j-1} \stackrel{\star}\longrightarrow B_{i+j-2} \stackrel{q}\longrightarrow H_{i+j-2}
      \end{equation*}
    makes  the shifted module $H_{*}[2]$ into a graded Lie algebra. 
\item The induced map $T: H_*[2]\to B_*[1]$ is a graded Lie algebra morphism. 
 Here, $B_*[1]$ is equipped 
    with the graded Lie algebra structure underlying its BV-algebra structure.
  \end{enumerate}
\end{lem}

Note that, since $T$ is an operator 
of odd degree, following the Koszul-Quillen sign convention, the bracket in statement (1) is 
given by $$\{x,y\}:= (-1)^{|x|} q\big(T(x)\star T(y)\big).$$

\begin{proof}  
The proof of statement (1) is the same as the proof of Theorem 6.1 in~\cite{CS}.

\smallskip

The Lie bracket $\{-,-\}_{\Delta}$ on the (shifted) modules $B_*[1]$ is defined by the degree 1 
operator (from $B_*\otimes B_*$ to $B_*$)
$$ \{a,b\}_{\Delta}:=(-1)^{|a|}\Delta(a\star b) -(-1)^{|a|}\Delta(a)\star b -a\star \Delta (b).$$
We denote the shift operator ($x\mapsto (-1)^{|x|} x$) by $s: B_*\to B_*[1]$. The Lie bracket on $B_*[1]$ 
is, by definition, the transport along $s$ of  the degree 1 operator above.  Now, 
for $x,y\in H_{*}$, since $T\circ q = \Delta$ and $\Delta \circ T=T\circ (q\circ T)=0$,  we deduce
from the above formula for   $\{-,-\}_{\Delta}$ that
  \begin{eqnarray*}
     T\big(\{x,y\} \big) &=& (-1)^{|x|} \Delta \big(T(x)\star T(y)\big) \\
              &=& - \,\{T(x),T(y)\}_{\Delta}\\
              &=& s^{-1}\{ s(T(x)), s(T(y))\}_{\Delta}.
    \end{eqnarray*}
This proves that $s\circ T$ is a Lie algebra map. 
\end{proof}

\begin{rmk}
We will apply Lemma~\ref{L:BVtoLieviaGysin} in the context of string topology  operations
(following~\cite{CS}). However, this Lemma also applies when 
$B_*$ is the Hochschild cohomology of any Frobenius algebra and
 $H_*$ is its negative cyclic cohomology 
(for instance see~\cite{Tradler, ATZ, Me}).
\end{rmk}

\begin{ex}
 Lemma \ref{L:BVtoLieviaGysin} also applies in the following situation. Let 
$C_*$ be the graded module $C_*=\bigoplus_{n\geq 0}(A[-1])^{\otimes n}$, where $A$ is a unital 
associative (possibly differential graded) algebra. In other words, as a 
$\bbZ$-graded module, 
$C_k=CHoch_{-k}(A)$, where $(CHoch_*(A),b)$ is the standard Hochschild chain 
complex~\cite{Loday}. Let 
$B:C_*\to C_{*-1}$ be the Connes operator, which makes $(C_*,D)$ into a chain complex. 
Since the Connes operator $D$ is a derivation for the shuffle product (see~\cite{Loday}),
the shuffle product makes  $sh: C_*\otimes C_* \to C_*$  into a differential graded commutative
algebra, and its homology $B_*:=H_*(C_*,B)$ into a graded commutative  algebra. Since $A$
is not necessarily commutative, the standard Hochschild differential $b:C_*\to C_{*+1}$ 
is not necessarily a derivation with respect to the shuffle product, but it is a second order
differential-operator. Thus $(C_*,B,b)$ is a differential graded \BV-algebra and, 
consequently, $(B_*,sh,b)$ 
is a \BV-algebra.

Let $NC_k:= \prod_{i\geq 0} C_{k+2i}$. It is easy to check that $(NC_k,B,(-1)^k b)$
 is a bicomplex (which can be thought as an analogue of the standard  cyclic chain complex 
where the role of $b$ 
and $B$ have been exchanged). Let $TC_*$ be the associated total complex of $NC_*$,
and let $H_*=H_*(TC_*)$ be its homology.
 The inclusion $q:C_k\hookrightarrow \prod_{i\geq 0} C_{k+2i}=TC_k$
is an injective chain map and, further, its cokernel is $TC_*[2]$. Let $T:TC_k\to C_{k+1}$
be the composition 
$$T: TC_k=\prod_{i\geq 0} C_{k+2i} \stackrel{projection}\twoheadrightarrow C_k \stackrel{b}\to C_{k+1}.$$
One checks easily that $T$ is a chain map; in fact, it is 
the connecting homomorphism of the short exact 
sequence $0\to C_*\to TC_*\to TC_*[2]\to 0$.

It follows that $H_*, B_*$ satisfy the assumption of Lemma~\ref{L:BVtoLieviaGysin}. Thus,
$H_*[2]$ inherit a natural Lie algebra structure.
\end{ex}

\subsection{Quick review of string topology operations for stacks}

Let $\XX$ be an oriented Hurewicz stack. 
It is shown in~\cite{BGNX} that  $H_*(\LXX)$ carries a natural structure 
of a $d$-dimensional Homological Conformal Field Theory, where $d=\dim\XX$. 
Restricting this structure to genus $0$-operations, one obtains the following.

\begin{thm}[\cite{BGNX}, Theorem 13.2] \label{T:BVforLoop}
Let $\XX$ be an oriented Hurewicz stack of dimension  $d$. Then, the shifted 
homology $(H_{i+d}(\LXX), \star, D)$
is a BV-algebra, where $D$ is the operator~\eqref{eq:DefDelta} 
induced by the $S^1$-action on $\LXX$ and 
$\star:H_i(\LXX)\otimes H_{j}(\LXX)\to H_{i+j-d}(\LXX)$ is the loop product.
\end{thm}

Note that, in general, the multiplication $\star$ may \emph{not} be unital  for stacks.

\subsection{Lie algebra structure on the $S^1$-equivariant homology of the free loop stack}

\begin{prop}\label{P:LieS1Stacks}
Let $\YY$ be an $S^1$-stack, and $d$ an integer.
Assume that the operator $D$ of~\eqref{eq:DefDelta}
induces a BV-algebra structure $(H_{i+d}(\YY), \star, D)$ on the (shifted) homology of $\YY$.
Then, we have the following:

\begin{itemize}
 \item The composition 
   \begin{multline*} \{-,-\} \: H_{i+d-2}^{S^1}(\YY)\otimes H_{j+d-2}^{S^1}(\YY)\stackrel{T\otimes T}\longrightarrow  
    H_{i+d-1}(\YY)\otimes  H_{j+d-1}(\YY)\\ 
    \stackrel{\star}\longrightarrow H_{i+j+d-2}(\YY) \stackrel{q_*}\longrightarrow H_{i+j+d-2}^{S^1}(\YY)
  \end{multline*}
 makes  the shifted equivariant homology $H_{*+d-2}^{S^1}(\YY)$ into a graded Lie algebra.
\item The induced map $T: H_*^{S^1}(\YY)[2]\to H_*(\YY)[1]$ is a graded Lie algebra morphism. Here, $H_*(\YY)[1]$ is equipped 
    with the graded Lie algebra structure underlying its BV-algebra structure.
\end{itemize}
\end{prop}

Recall the sign convention for bracket and similarly for higher brackets in  statement (2).

\begin{proof} 
 By Proposition~\ref{P:GysinSequence} and Lemma~\ref{L:DgivenbyGysin}, 
 the shifted equivariant homology $H^{S^1}_{*+d}(\YY)$ and shifted homology
 $H_{*+d}(\YY)$ satisfy the assumption of  
 Lemma~\ref{L:BVtoLieviaGysin}.  
\end{proof}

Let $\XX$ be a topological stack. Then, the free loop stack $\LXX=\Map(S^1,\XX)$ is a topological stack
 (Theorem~\ref{T:mapping}) with a (strict) $S^1$-action (see Section~\ref{SS:mappingaction}). 
Further, if $\XX$ is a Hurewicz (for instance differentiable) oriented stack, then, by Theorem~\ref{T:BVforLoop}, 
its (shifted down by $\dim\XX$) homology carries a structure of a \BV-algebra. Hence, we can apply
 Proposition~\ref{P:LieS1Stacks} to $\YY=\LXX$.

\begin{cor}\label{C:LieLoopStack}
 Let $\XX$ be an oriented differentiable (or more generally, Hurewicz) stack of dimension $d$.
\begin{itemize}
 \item For $x,y\in H_*(\LXX)$,  the formula
           $$\{x,y\}:= (-1)^{|x|} q\big(T(x)\star T(y)\big)$$ 
makes the equivariant homology $H_*^{S^1}(\LXX)[2-d]$ into a graded Lie algebra. Here,
$T$ is the transgression map (Definition~\ref{D:transfer}) and 
$q:\LXX \to [S^1\backslash \LXX]$ is the canonical projection.

\item The transgression map $T: H_*^{S^1}(\LXX)[2-d]\to H_*(\LXX)[1-d]$ 
is a graded Lie algebra homomorphism. Here, $H_*(\LXX)[1-d]$ 
   is the graded Lie algebra structure underlying the \BV-algebra 
   structure of Theorem~\ref{T:BVforLoop}.
\end{itemize}
\end{cor}
The bracket $\{-,-\}$ defined in Corollary~\ref{C:LieLoopStack} is called the \textbf{string bracket}. 

\subsection{Some Examples}{\label{SS:examples}}

\begin{ex}[Oriented manifolds] \label{E:manifold}
 Let $M$ be an oriented closed manifold. Then, by~\cite[Proposition 17.1]{BGNX}, the 
 \BV-algebra structure of $H_*(LM)$ given by Theorem~\ref{T:BVforLoop} agrees 
 with Chas-Sullivan construction (and other constructions as well).  Since the 
 $S^1$-action on $LM$ agrees with the stacky one (\cite[Example 5.8]{BGNX}), 
 it follows immediately that the Lie algebra structure given by Corollary~\ref{C:LieLoopStack} 
 agrees with the one in Chas-Sullivan~\cite{CS} for oriented closed manifolds. Note that 
 Corollary~\ref{C:LieLoopStack} also applies to open oriented manifolds. 
\end{ex}

\begin{ex}[Classifying stack of compact Lie groups]
Let $G$ be a compact  Lie group. Its associated classifying stack $[G\backslash *]$ 
is oriented (see~\cite{BGNX}) of dimension $-\dim G$, hence its $S^1$-equivariant 
homology $H^{S^1}_*(\Lo [G\backslash *])$ has a degree $2+\dim G$ Lie bracket.

\begin{prop}
If $\kor$ is of characteristic zero and $G$ is  either connected or 
finite, then  
the Lie algebra $H_{*}(\textrm{L}[G\backslash *], \kor)$ is abelian. 
\end{prop}

\begin{proof}
By~\cite[Theorem 17.23]{BGNX}, if $G$ is connected, the hidden loop product 
(which coincides with the loop product by~\cite[Lemma 17.14]{BGNX}) vanishes. 
Thus, the string bracket vanishes as well. 

If $G$ is finite, then $H_{*}(\textrm{L}[G\backslash *], \kor)$ is concentrated in degree 
$0$ and it follows that the transfer map 
$T:H_{*}^{S^1}(\textrm{L}[G\backslash *], \kor)\to H_{*+1}(\textrm{L}[G\backslash *], \kor)$ vanishes, hence so does the string bracket.
\end{proof}
\end{ex}

If $G$ is a finite group with order coprime to the characteristic of $\kor$, then 
the same proof shows that $H_{*}(\textrm{L}[G\backslash *], \kor)$ is abelian. 

However, if the characteristic of $\kor$ divides the order of $G$, then, 
in view of the results of~\cite{SF} on the nontriviality of the Gerstenhaber bracket on
the Hochschild cohomology of the group algebra $\kor[G]$ of $G$, and the close 
relationship between the Gerstenhaber product and loop bracket~\cite{FT, GTZ3}, 
it is reasonable to expect that the string Lie algebra of $[G\backslash *]$ is 
no longer abelian in this case. 

\begin{ex}[A non-nilpotent example]
 Let $\sz$ acts on the euclidean sphere 
   $$S^{2n+1}=\{
|z_0|^2+\cdots+|z_{n}|^2=1, z_i\in \mathbb{C}\}$$ 
as the group generated by the reflections  across the
hyperplanes  $z_i=0$, $0\leq i\leq n$. Let $\TT=[S^{2n+1}\times S^{2n+1}/\sz]$ be
the induced quotient stack, where $\sz$ acts diagonally. This is  an oriented orbifold 
in the sense of~\cite{BGNX}. 

Recall that there is an isomorphism of coalgebras 
$H_*^{S^1}(*)\cong H_*(BS^1)\cong \kor[u]$, where $|u|=2$. Thus,
$H_*^{S^1} (\textrm{L} (S^{2n+1}\times S^{2n+1}),\kor)$ is a $k[u]$-comodule, and
$H_*^{S^1}(\textrm{L}\TT, \kor)$ is a $\kor[\sz][u]$-comodule, where
$\kor[\sz][u]$ is the coalgebra obtained by tensoring $k[u]$ with the group
algebra $\kor[(\mathbb{Z}/2\mathbb{Z})^{n+1}]$.

\begin{prop}\label{P:loopbracketnontrivial}
 Let $\kor$ be a field of characteristic
  different from $2$. 
  \begin{itemize}
\item[i)] There is an isomorphism of (graded)  Lie algebras 
$$H_*^{S^1}(\textrm{L}\TT, \kor) \, \cong \, 
   H_*^{S^1} (\textrm{L} (S^{2n+1}\times S^{2n+1}),\kor)\otimes_{\kor}
  \kor[(\mathbb{Z}/2\mathbb{Z})^{n+1}].$$
   \item[ii)] As a  $\kor[\sz][u]$-comodule,
 $H_*^{S^1}(\textrm{L}\TT, \kor) $ is  free and is spanned by the basis elements
 $$\big(e_{i,j}\big)_{(i,j)\in \mathbb{N}^2\setminus\{(0,0)\}}, 
      \qquad  \big(f_{i,j}\big)_{(i,j)\in \mathbb{N}^2} $$
where $|e_{i,j}|=2n(i+j)$ and $|f_{k,l}|=2n(i+j+2)+1$.
 \item[iii)] The string bracket is $\kor[\sz][u]$-colinear and satisfies the formulas
 \begin{align*}
  [f_{i,j}, e_{k,l}] &= \binom{i+k}{i} \binom{j+l}{j} 
  \frac{il-jk}{(i+k)(j+l)} f_{i+k-1, j+l-1}, \\
  [e_{i,j}, e_{k,l}] &= \binom{i+k}{i} \binom{j+l}{j} 
  \frac{jk-il}{(i+k)(j+l)} e_{i+k-1, j+l-1},\\
  [e_{i,j}, e_{k,l} ] &=0.
 \end{align*}
\end{itemize}
\end{prop}

Since $[e_{1,1}, e_{i,j} ] = (i-j) e_{i,j}$, it follows that 
$H_*^{S^1}(\textrm{L}\TT, \kor) $ is not nilpotent as a Lie algebra.
 
\begin{proof}
The explicit computations in (ii) and (iii) follow from (i) 
and the standard computations of equivariant homology of loop spheres; see~\cite{FTV, BV}. 

The statement (i) follows from~\cite[\S~17]{BGNX}, as we now 
explain. By~\cite[Proposition 5.9]{BGNX}, 
the free loop stack $\textrm{L}\TT$ is presented by the transformation topological groupoid
$$LT:=\big[\coprod_{g \in R} \mathcal{P}_g \rtimes \sz 
\toto \coprod_{g \in R} \mathcal{P}_g\big], $$ 
where $\mathcal{P}_g$ is the space of continuous maps 
  $$\mathcal{P}_g:=\{f: \bbR\to S^{2n+1} \times S^{2n+1},\, 
  f(t) =f(t+1)\cdot g \textrm{ for all } t\}.$$ 
The  $\sz$ action on $\mathcal{P}_g$ is pointwise.  
The action of $S^1$, or rather $[\bbZ\backslash \bbR]$, on $\textrm{L}\TT$ 
is presented by the morphism of topological groupoids 
  $$\big(\bbZ\times \bbR \big) \times 
  \big(\coprod_{g \in R} \mathcal{P}_g \rtimes \sz\big)  
  \stackrel{\theta}\longrightarrow \big(\coprod_{g \in R} \mathcal{P}_g \rtimes \sz\big)$$
defined, for  $(n,x)\in \mathbb{Z}\times \mathbb{R}$, $f\in \mathcal{P}_g$, and
$h\in \sz$ by
\begin{align*}{\theta}(x,n,f,h) & = \big((t\mapsto f(t+x)\big), g^nh\big).
\end{align*} 
The map is compatible with the group structure of the stack 
$[\bbZ\backslash \bbR]$, hence is a groupoid morphism presenting the
$S^1$-action on $\LXX$ defined in Section~\ref{S:actionmapping}.

Since $\sz$ is a subgroup of the connected Lie group $SO(2n+2)$, which acts  
diagonally on $S^{2n+1}\times S^{2n+1}$,  for all $g\in \sz$ there is a 
continuous path $\rho: [0,1]\to SO(2n+2)$ connecting $g$ to the identity 
(that is $\rho(0)=g$, $\rho(1)=1$). This allows us to define a map  
  $$\Upsilon: \coprod_{g\in \sz} \mathcal{P}_g \to \coprod_{g\in \sz} 
    {\rm L}\big(S^{2n+1}\times S^{2n+1}\big),$$ 
which is given, for any path $f\in P_g$, by the loop 
$$
\Upsilon_g(f)(t) = \left\{ \begin{array}{ll} f(2t) & \textrm{if } 0\leq t\leq \frac{1}{2}, \\
 f(0)\cdot\rho(2t-1) &  \textrm{if } \frac{1}{2} \leq t\leq 1. \end{array}\right.
$$
It is a general fact  that $\Upsilon$ is  a $\sz$-equivariant  homotopy equivalence, 
where $\sz$ acts pointwise  
(see~\cite[\S~6]{LuUrXi} for details); note that this 
action is trivial in homology with coefficients in a field of characteristic coprime 
to $2$ since the (naive) quotient map 
$$S^{2n+1}\times S^{2n+1}\to \sz\backslash(S^{2n+1}\times S^{2n+1})
\cong S^{2n+1}\times S^{2n+1}$$ 
is invertible in homology in this case.  It follows that $\Upsilon$ induces an isomorphism  
  $$H_*(\textrm{L}\TT)\cong 
  H_*\big(\textrm{L}(S^{2n+1}\times S^{2n+1}),\kor\big)\otimes_{\kor} \kor[\sz]$$ 
and, by Corollary~\ref{C:specseq},  an isomorphism of $\kor[u]$-comodules 
 \begin{equation} \label{eq:LSmodZ2} H_*^{S^1}(\textrm{L}\TT, \kor) \, 
  \cong \, H_*^{S^1} (\textrm{L} (S^{2n+1}\times S^{2n+1}),\kor)\otimes_{\kor}
  \kor[(\mathbb{Z}/2\mathbb{Z})^{n+1}]. 
 \end{equation}
The proof that the above isomorphism is multiplicative with respect to the 
loop product is similar to the proof of~\cite[Proposition 17.10]{BGNX}. 
Furthermore, by naturality of the Gysin sequence (Proposition~\ref{P:GysinSequence}), 
the Gysin sequence of the $S^1$-stack $\textrm{L}\TT$ is identified with the Gysin 
sequence of the $S^1$-stack 
$$\coprod_{g\in \sz} [\sz\backslash \textrm{L}(S^{2n+1}\times S^{2n+1})],$$ 
where the $[\bbZ\backslash\bbR]$-action is induced by the map $\Upsilon$. 
By definition of the Lie algebra structure, it follows that the 
isomorphism~\eqref{eq:LSmodZ2} is an isomorphism of Lie algebras (after  
shifting the degree by $\dim\XX-2$).
\end{proof}
\end{ex}

\section{Functoriality of with respect to open embeddings}{\label{S:Lemma}}

In this section we show that the Batalin-Vilkovisky structure and the string bracket are 
functorial with respect to open 
embeddings (Proposition \ref{P:functoriality}). 
This will be used later  to describe the string bracket for 2-dimensional orbifolds.

\begin{lem}{\label{L:tubular}}   
 Let $\XX$ be a topological stack whose coarse moduli space $|\XX|$ is paracompact. 
 Let $\FF$ be a metrizable vector bundle over $\XX$, and let $\UU \subset \FF$ be 
 an open substack of the total space of $\FF$ through which the zero 
 section $s \:\XX \to \FF$ factors. Then, the map $s \: \XX \to \UU$ admits 
 a tubular neighborhood \cite[Definition 8.5]{BGNX}. That is, there is a  vector 
 bundle $\N$ over $\XX$ and a factorization   
      $$\XX  \stackrel{i}{\hookrightarrow} \NN  \stackrel{j}{\hookrightarrow}   \UU$$
 for $s$, where $i$ is the zero section of  $\NN$ and $j$ is an open embedding.
\end{lem}

\begin{proof}
 We show that there is a function $f \: \XX \to \mathbb{R}^{>0}$ such that the map
 $\Phi \: \FF \to \FF$ defined by fiberwise multiplication by $f$ identifies  
 the open unit ball bundle $\DD \subset \FF$ with an open substack 
 $\VV \subseteq \UU$. It would then follow that $\VV$ is isomorphic, as a 
 stack over $\XX$, with $\DD$, which is in turn isomorphic to the total space of $\FF$. 
 Thus, taking $\NN:=\FF$ gives the desired factorization.
 
 Since we have partition of unity on $|\XX|$, construction of $f \: \XX \to \mathbb{R}^{>0}$ 
 can be done locally on $|\XX|$,  so we are allowed to pass to open substacks of $\XX$. 
 Thus, we may assume that $\XX$ admits a chart $\pi \: X \to \XX$ 
 such that after base extending along $\pi$, the resulting bundle $F$ over $X$ and 
 the open set $U \subseteq F$  corresponding to $\UU$ have the property that $U$ contains  
 an $\varepsilon$-ball bundle of $F$, for some $\varepsilon >0$. So, it is enough to take 
 $f \: \XX \to \mathbb{R}^{>0}$ to be the constant function  $\varepsilon$.
\end{proof}

\begin{lem}{\label{L:nnspullback}}   
   Consider the 2-cartesian diagram of topological stacks
             $$\xymatrix@=26pt@M=8pt{ 
                \XX \ar[r]_{f}^*+<4pt>[o][F-]{\theta_{f}} \ar[d]& \YY \ar[d] \\
                  \XX' \ar[r]_{f'}^*+<4pt>[o][F-]{\theta_{f'}} & \YY'    }$$
     in which the vertical arrows are open embeddings. If $f'$ is  
    bounded proper (respectively, normally nonsingular),
     then so is $f$ (see~\cite[Definitions 6.1, 8.15]{BGNX}). Suppose,
     in addition, that
     $f$ and $f'$ are strongly proper (see~\cite[Definition 6.2]{BGNX}), and let 
     $\theta_f$ and $\theta_{f'}$ be the corresponding strong orientation classes 
     \cite[Proposition 8.25]{BGNX}. Then, $\theta_{f}$ is the independent pullback of
     $\theta_f'$ in the sense of bivariant theory~\cite[\S~7.2]{BGNX}. 
\end{lem}

\begin{proof}
  Being bounded proper is invariant under arbitrary base change. Suppose that $f'$ is
  normally nonsingular, and let
        $$\xymatrix@=16pt@M=8pt{ \NN'  \ar@{^(->}^{i'} [r] & \EE' \ar[d]^{p'} \\
                          \XX'\ar[u]^{s'} \ar[r]_{f'} & \YY'     }$$ 
   be a normally nonsingular diagram for it.  Base changing the diagram along the open 
   embedding $\YY \to \YY'$,
  we obtain a  diagram
              $$\xymatrix@=16pt@M=8pt{ \UU  \ar@{^(->}^{i} [r] & \EE \ar[d]^{p} \\
                          \XX\ar[u]^s \ar[r]_{f} & \YY     }$$    
  where $\UU=(p'\circ i')^{-1}(\YY)$ is an open substack
 of  the vector bundle $\FF:=\NN'|_{\XX}$ over $\XX$ which contains the 
 zero section $s \: \XX \to \FF$. This diagram is not quite a  normally nonsingular diagram,
 as $\UU$ is not a vector bundle over $\XX$, but by Lemma \ref{L:tubular} the map 
 $s \: \XX \to \UU$ admits a tubular  neighborhood $\NN$. Replacing $\UU$ by his 
 tubular neighborhood $\NN$ we obtain the desired normally nonsingular diagram for $f$.
 
 The statement about $\theta_{f}$ being the independent pullback of
 $\theta_f'$ follows from the definition of independent pullback~\cite[\S~7.2]{BGNX}
 and excision.
\end{proof}

\begin{prop}{\label{P:functoriality}}
  Let $\XX$ be an oriented  Hurewicz stack of dimension $d$, and $\UU \subseteq \XX$ 
  an open substack.   Then, $\UU$ inherits a natural orientation  from $\XX$, and the 
  induced map $H_{*+d}(L\UU) \to H_{*+d}(\Lo\XX)$  is a morphism of BV-algebras. Therefore, 
  the induced map  $H_*^{S^1}(L\UU)[2-d] \to H_*^{S^1}(\Lo\XX)[2-d]$ is a morphism of graded 
  Lie algebras.
\end{prop}

The Proposition  applies, in particular, to an embedding of oriented manifolds 
(see Example~\ref{E:manifold}).

\begin{proof}
  To prove that $\UU$ inherits an orientation, we have to show that 
  $\Delta_{\UU} \: \UU \to \UU\times\UU$ is   strongly 
  proper~\cite[Definition 6.2]{BGNX}, normally nonsingular, and   that the class 
  $\theta_{\UU} \in H^d(\Delta_{\UU})$ obtained by pulling back the strong orientation 
  class   $\theta_{\XX} \in H^d(\Delta_{\XX})$ via independent pullback, as in the 
  2-cartesian diagram
      $$\xymatrix@=26pt@M=8pt{ 
            \UU \ar[r]_{\Delta_{\UU}}^*+<4pt>[o][F-]{\theta_{\UU}} \ar[d] & 
             \UU\times \UU \ar[d] \\
             \XX \ar[r]_{\Delta_{\XX}}^*+<4pt>[o][F-]{\theta_{\XX}} & \XX\times \XX   }$$
  is a strong orientation~\cite[Definition 8.21]{BGNX}.  These all follow 
  from Lemma \ref{L:nnspullback}, except the fact
  that $\Delta_{\UU}$ is strongly proper (the lemma only says that it is bounded proper). 
  The fact that $\Delta_{\UU}$ is strongly proper follows from the observation made 
  in~\cite[Example 6.3.2]{BGNX}.

  Let us now prove that the loop product is preserved under the  map 
  $H_{*+d}(L\UU) \to H_{*+d}(\Lo\XX)$. By the
  construction of the loop product~\cite[\S~10.1]{BGNX}, proving this reduces 
  to showing that in the    2-cartesian diagram
       $$\xymatrix@=26pt@M=8pt{ 
         \Map(8,\UU) \ar[r]^*+<4pt>[o][F-]{\omega} \ar[d]_g & L\UU\times L\UU \ar[d]^f \\
            \Map(8,\XX) \ar[r]_*+<4pt>[o][F-]{\theta} & \Lo\XX\times \Lo\XX    }$$
 the  Gysin maps $\theta^!$ and $\omega^!$ are compatible, in the sense that 
 \begin{eqnarray}\label{eq:Gysin}
      g_*(\omega^!(c))=\theta^! (f_*(c)), \ \  \text{for every} \ c \in H_*(L\UU\times L\UU).
\end{eqnarray}
  Here, the bivariant class $\theta$ is the one obtained via independent 
  pullback from the strong orientation
  class  $\theta_{\XX} \in H^d(\Delta_{\XX})$, as in the diagram
         $$\xymatrix@=26pt@M=8pt{ 
              \Map(8,\XX) \ar[r]^*+<4pt>[o][F-]{\theta} \ar[d]& \Lo\XX\times \Lo\XX \ar[d] \\
             \XX \ar[r]_{\Delta_{\XX}}^*+<4pt>[o][F-]{\theta_{\XX}} & \XX\times\XX    }$$
  (Similarly, the class $\omega$ is obtained from  the strong orientation
  class $\theta_{\UU} \in H^d(\Delta_{\UU})$.)     
  
  To prove the compatibility relation (\ref{eq:Gysin}), we note that, by what we just 
  showed in the first part of the proof,   the bivariant class $\omega$ is the 
  independent pullback  of $\theta$. Hence, the relation (\ref{eq:Gysin})
  follows from the `Naturality' of Gysin maps~\cite[\S~9.2]{BGNX}.    
\end{proof}

\section{Goldman bracket for $2$-dimensional orbifolds and stacks}{\label{S:Goldman}}
 
By Corollary \ref{C:LieLoopStack}, when $\XX$ is an oriented 2-dimensional 
Hurewicz stack, the  equivariant homology $H_*^{S^1}(\LXX)$ is a graded Lie algebra.   
When $\XX=X$ is an honest surface, it is well-known that the degree $0$
part  $H_0^{S^1}(\Lo X)$ is freely generated by the homotopy classes of free loops on $X$,
and the Lie bracket is the famous Goldman bracket (from~\cite{Go}). The above 
relationship between equivariant homology and free loops 
holds for general stacks  as well (see Lemma~\ref{L:surjectionpi1} below).

\begin{defn} The \emph{Goldman bracket} of  an oriented $2$-dimensional stack $\XX$    
is the restriction to $H_0^{S^1}(\Lo\XX)$ of the (degree $0$) Lie bracket on 
$H_*^{S^1}(\Lo\XX)$.
\end{defn}

In this section, we describe in detail the case were 
$\XX$ is a reduced (or effective) oriented 2-dimensional orbifold.
The functoriality lemma proved in the previous section allows us to explicitly 
write down the Goldman bracket for 
such $\XX$. The idea is that the inclusion
$\UU \hookrightarrow \XX$ of the complement
of the orbifold locus of $\XX$ induces a surjection on fundamental 
groups $\pi_1(\UU) \to \pi_1(\XX)$. Therefore,
thanks to the functoriality (Proposition~\ref{P:functoriality}) and the following lemma, 
to compute the bracket of free loops in $\XX$, we can  first lift them to free loops 
in $\UU$, compute the bracket there, and then project back down to $\XX$ 
(see Lemma~\ref{L:bracketreducedorbifold} below).

\begin{lem}\label{L:surjectionpi1}
   Let $\XX$ be a topological stack. Then, we have a natural isomorphism 
      $$H_0^{S^1}(\Lo\XX)=k[C],$$  
   where  $C$ is the set of free homotopy classes of  loops on $\XX$, and $k$ 
   is the coefficients  of the homology. When $\XX$ is connected,
   $C$ is equal to the set $\Conj(\pi_1\XX)$ of conjugacy classes in
   $\pi_1\XX$.
\end{lem}

\begin{proof}
   The result is standard when $\XX=X$ is an honest topological space. We 
   reduce the general stacks    to this case as follows. 
   By definition, $H_0^{S^1}(\Lo\XX)=H_0 [S^1\backslash \Lo\XX]$. By Corollary 
   \ref{C:invariance},    $H_0 [S^1\backslash \Lo\XX]\cong
   H_0 [S^1\backslash \Lo X]=k[C']$, where $\varphi \: X \to \XX$ is a classifying 
   space for $\XX$ and    $C'$   is the set of free homotopy classes of  loops 
   on $X$. Since $\varphi$ induces a bijection  $C' \risom C$, the result follows. 
\end{proof}
 
\subsection{Goldman bracket for reduced  2-dimensional orbifolds}\label{SS:GoldmanOrbifold}

First we look at the case of reduced orbifolds.
Let $\XX$ be a 2-dimensional reduced oriented orbifold, and let $\UU \subseteq \XX$ 
be the complement of the orbifold locus. Recall that a reduced $2$-dimensional 
orbifold is a surface together with a discrete set of orbifold points. Each such 
orbifold point  $x$ has an isotropy group which is a cyclic group 
$\bbZ\slash n\bbZ$, and the complement 
of the orbifold locus is a surface with a discrete set of punctures.

The following lemma is a consequence of van Kampen,  Lemma~\ref{L:surjectionpi1} and 
Proposition~\ref{P:functoriality}.
 
\begin{lem}\label{L:bracketreducedorbifold}
 Let $\XX$ be a 2-dimensional orbifold (not necessarily reduced), 
 and let $\UU \subseteq \XX$ be the 
 complement a finite set of points. The natural map $H_0^{S^1}(\Lo\UU) \to H_0^{S^1}(\LXX)$ 
 is a surjective map of Lie algebras.
\end{lem}

The above lemma allows us to compute the string bracket of a reduced orbifold by lifting  
free loops to the complement of the orbifold locus and computing the 
string bracket there using the usual intersection theory 
of curves on an honest surface. 

\subsection{Explicit description for disk with orbifold points}{\label{SS:explicitGoldman}}

As a trivial example, consider a disk with an orbifold point with isotropy group 
$\bbZ\slash n\bbZ$. Then, the Goldman Lie algebra is the free $\kor$-module spanned 
by the set $\bbZ\slash n\bbZ$. It is an abelian Lie algebra since every two loops 
around the orbifold point can be homotopically deformed so that 
they do not intersect.  We describe below the general case of a disk with more than one
orbifold points. We start with the case of two orbifold points.

\subsubsection{The disk with two orbifold points}\label{SS:diskwithtwoorbipoints}

Consider a disk $\DD$ with two orbifold points $x$ and $y$. 
Let $a\in \bbZ\slash n\bbZ$ and  $b\in \bbZ \slash m\bbZ$ be the generators 
of the isotropy groups at the points $x$ and $y$, respectively. 
By van Kampen, the fundamental  group of $\DD$ at a chosen based point, different from
$x$ and $y$,  is  isomorphic 
to the free product of the isotropy groups of the orbifold points,   
                 $$\pi_1(\DD) \,\cong\, \bbZ\slash n \bbZ * \bbZ\slash m \bbZ,$$ 
so that every free loop is given  by a (cyclic) word in the generators $a$ and $b$,  as in 
Figure~\eqref{Fig:presentationofalpha}. Since $a$ and $b$ have finite order, we do not need to 
consider negative powers of $a$ or $b$ to present a loop.  

\begin{figure}\begin{center}
\includegraphics[scale=0.6]{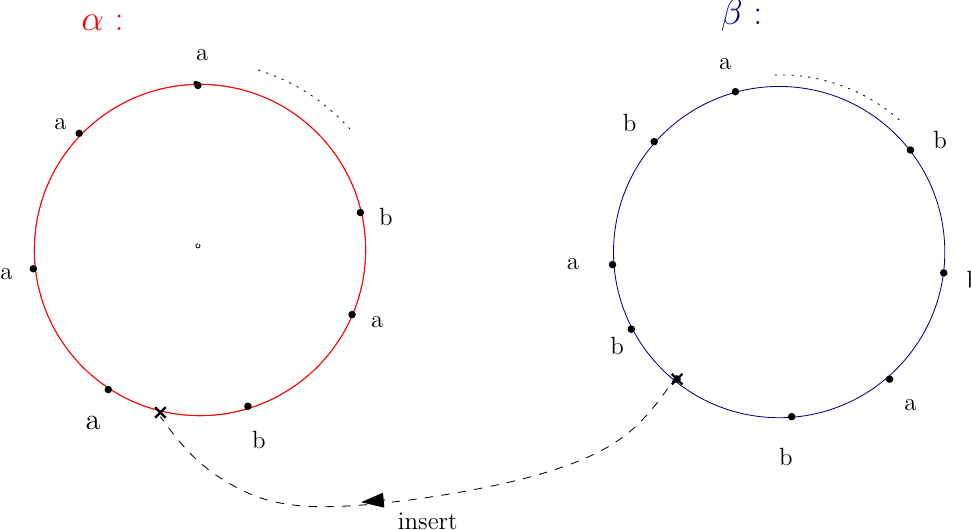}
\caption{A presentation of $\alpha$ and  $\beta$ with an admissible insertion}\label{Fig:presentationofalpha}\end{center}
\end{figure}

We now describe the Goldman bracket. Let $\alpha$ and $\beta$ be 
two free loops presented by cyclic words as in Figure~\eqref{Fig:presentationofalpha}, 
that is as circle with finitely many points labelled 
by either $a$ or $b$; the intervals between these points are colored  red  for $\alpha$ 
and blue  for $\beta$.

We define the bracket $\{\alpha, \beta\}$ as follows.  
\begin{enumerate}
 \item \label{eq:step1} Determine all \emph{admissible pairs}, that is  
   the pairs consisting  of a red and  a blue interval such that
 \begin{enumerate}
  \item the end points in the blue interval have the same labels;
  \item the end points in the red interval have different labels.
 \end{enumerate}
 \item \label{eq:step2} For each admissible pair, cut both circles in the middle of the 
  chosen intervals and insert the blue circle into the red one by joining the cut 
  intervals and   preserving the cyclic ordering. 
 \item \label{eq:step3} Assign the sign $+$ to the new circle obtained in Step~\ref{eq:step2} 
  if 
  \begin{itemize} \item the red interval is $ab$ (in the cyclic ordering) and the 
    blue interval is $aa$,
    \item or if the red interval is $ba$ (in the cyclic ordering) and the blue interval is $bb$
  \end{itemize}
Otherwise, assign the opposite sign.
 \item \label{eq:step4} Sum up all circles obtained in Step~\ref{eq:step1}, 
 for all admissible pairs,  with signs  
 as in   Step~\ref{eq:step3}. This sum is our bracket $\{\alpha, \beta\}$.
\end{enumerate}
In the above procedure, a loop given by a single generator is regarded as a loop with a 
single interval which has the same end points. It is straightforward to check that its 
bracket with any other loop is trivial. 

\begin{prop}
 The bracket $\{\alpha, \beta\}$ given by the above procedure is the Goldman bracket of 
 $\alpha$ and $\beta$ in $H_0^{S^1}(\Lo\DD)$.
\end{prop}

\begin{proof} 
By Lemma~\ref{L:bracketreducedorbifold}, $\alpha$ and $\beta$ can, respectively,  be
presented by a red and a blue loop, as in Figure~\eqref{Fig:loops2orbipoints}. 
These loops lie entirely in the surface $\DD\setminus\{x,y\}$, and their 
Goldman bracket is simply  their usual loop bracket  computed in $\DD\setminus\{x,y\}$. 

Note that, up to homotopy, we may assume  that  the red loop is sitting inside 
the blue loop in the neighborhood of $x$ and $y$ and the intersection between 
the loops $\alpha$, $\beta$ are transverse. Furthermore,  we may also assume that, 
when going from the neighborhood of a point  to the other, the blue loop and the 
red loop does not intersect along the way; see
  Figure~\eqref{Fig:loops2orbipoints}. 
\begin{figure} \begin{center}
\includegraphics[scale=0.7]{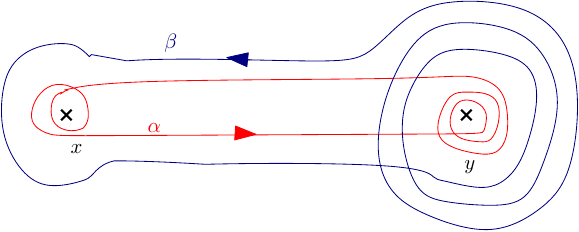}
\caption{A red loop $\alpha$ sitting inside a blue loop $\beta$ }\label{Fig:loops2orbipoints}\end{center}
\end{figure}

With these conventions, the blue loop and the red loop 
intersect only when the blue loop is making turns 
around one of the orbifold points $x$ and $y$, and the red loop is travelling 
from one orbifold point to the other;  see Figure~\eqref{Fig:loops2orbipoints}. 
Now we apply the usual formula for computing the Goldman bracket on the  
honest surface $\DD\setminus\{x,y\}$, which yields the result in 
Step~\ref{eq:step4} after  quotienting out the relations  $a^n=1$ and $b^m=1$.
\end{proof}

%
%
%

\begin{ex}
  Suppose $a^2=1$ and $b^4=1$. Let $\alpha:= a^2b$ and $\beta:= ab^2$. Note that
  $\alpha= a^2b=b$, so 
    $$\{\alpha, \beta\}=\{b, b^2a\}=0,$$
  because we can choose a very small representative for the loop $b$ which does 
  not intersect (a given representative of) the  loop $b^2a$. 
  
  On the other hand, if we follow our algorithm above, we get
  $$\{\alpha, \beta\}=bab^2a^2 - b^2aba^2.$$  
  It is not obvious that this is equal to zero. Using the relation $a^2=1$, we can reduce it to
      $$\{\alpha, \beta\}=bab^2 - b^2ab.$$ 
 The latter is zero because these are free loops so we are allowed to 
 perform cyclic permutation on the words.
\end{ex}

\begin{ex}\label{ex:nonzerodiskwithtwopoints}
  Suppose $a^3=1$ and $b^4=1$. Let $\alpha:= a^2b$ and $\beta:= ab^2$, as 
  in the previous example.   Using our algorithm above, we get
  $$\{\alpha, \beta\}=bab^2a^2 - b^2aba^2.$$  
 It is easy to see that this is indeed non-zero.
\end{ex}

\subsubsection{The disk with finitely many orbifold points}\label{SS:diskwithmanyorbipoints}
We now describe the general case of a disk $\DD$  with finitely many orbifold points 
$x_1,\dots, x_r$. Let $n_i$ be the order of the orbifold point 
$x_i$. Choose a simple loop $a_i$ going counterclockwise around $x_i$. This
identifies the isotropy group of $\DD$ at the point $x_i$ with 
$\bbZ\slash n_i \bbZ$. For simplicity of visualization, we assume that
the points $x_i$ are arranged on a circle in the same order. 

By van Kampen, the fundamental  group of $\DD$ at a chosen base point (different from all $x_i$) 
is  isomorphic to the free product of the isotropy groups of the orbifold points,   
                 $$\pi_1(\DD) \,\cong\, \bbZ\slash n_1 \bbZ * \cdots * \bbZ\slash n_r \bbZ,$$ 
and $a_i$, $1\leq i \leq r$, are its generators.

 As in the case of two orbifold 
points, we present free loops $\alpha$, $\beta$ by  (cyclic) words on 
the generators $a_i$, as in Figure~\eqref{Fig:presentationalphamanypoints}. 
Again, the intervals in $\alpha$ are colored red   and  the ones in $\beta$ are colored blue.

\begin{figure} \begin{center}
\includegraphics[scale=0.6]{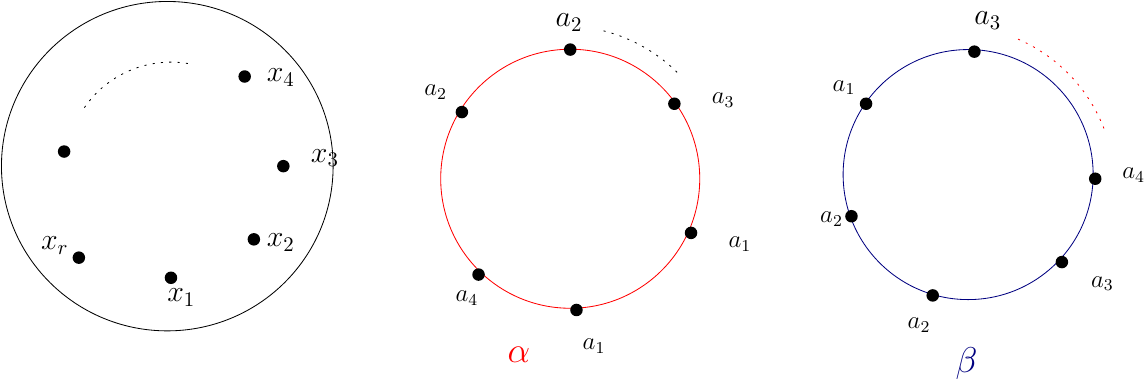}
\caption{A disk with $r$ orbifold points and the presentation of a red loop $\alpha$ and a blue 
loop $\beta$}\label{Fig:presentationalphamanypoints}
\end{center}
\end{figure}

The bracket $\{\alpha,\beta\}$ is given by a similar cut and insert procedure as 
in the case of two orbifold points. 
The only difference is in Step~\ref{eq:step1} and Step~\ref{eq:step3} where we determine 
which red and blue intervals are to be cut and what sign to assign after inserting the 
blue loop into the red loop. 

\begin{enumerate}
 \item   Determine the admissible pairs.  A pair consisting
 of a red interval $a_{i} a_{j}$ and a 
 blue interval $a_{k}a_{l}$
 is called   \textit{admissible} if it  satisfies the following conditions:
 \begin{itemize}
 
 \item We have that $i\neq j$, the (unoriented) intervals are distinct, and the 
 red segment $[ij]$   intersects the (possibly degenerate) 
 blue segment $[kl]$ (in the cyclic arrangement of the $x_i$; see Figure~\eqref{Fig:ijcutskl}).
\begin{figure} \begin{center}
\includegraphics[scale=0.5]{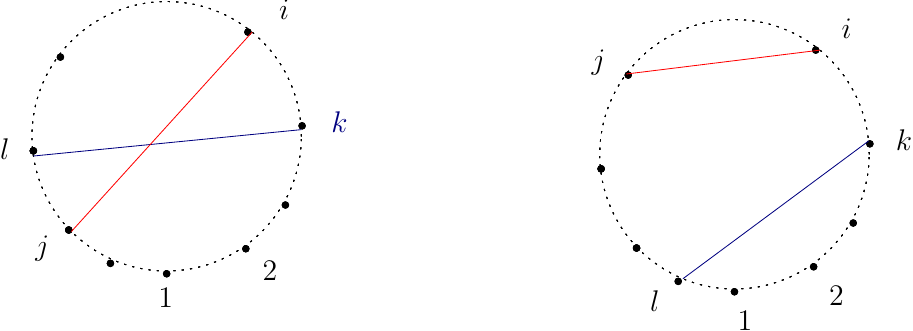}
\caption{The left pair is  admissible while the right pair is not admissible}\label{Fig:ijcutskl}
\end{center}

\end{figure}
  \item If either $k=l$ or all four points are distinct, then there is no 
  further condition. Otherwise, the red segment $[ij]$ and the 
  blue segment $[kl]$ share a vertex (or two). There are two possible cases:
\begin{itemize}

\item Case 1: the intervals intersect in $k$. 
  In this case, the pair is admissible if the red  segment is above
  the blue segment 
  in the cyclic ordering; see Figures~\eqref{Fig:segmentconditionforcut}. Otherwise
  it is not  admissible. 
 
\begin{figure} \begin{center}
\includegraphics[scale=0.5]{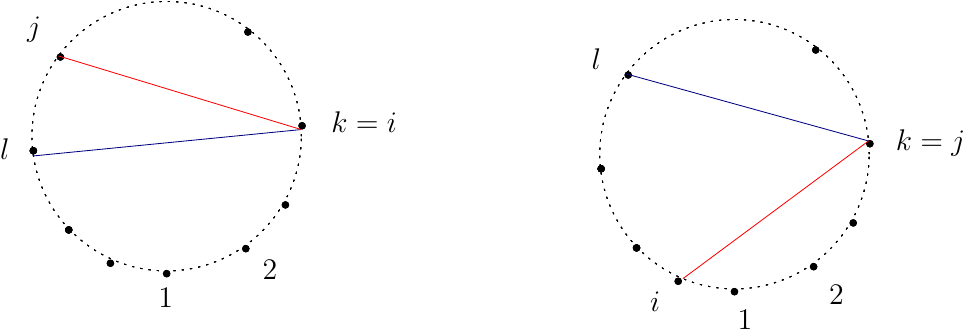}
\caption{The left triangle is an  admissible pair while the right triangle  is not}\label{Fig:segmentconditionforcut}
\end{center}
\end{figure}

\item  Case 2:  the intervals intersect in $l$. 
In this case, the pair is admissible if the red segment is below the blue segment in 
the cyclic ordering (and is not admissible otherwise). 
\end{itemize}\end{itemize}
 
\item  For every admissible pair as in Step~(1), cut the blue circle 
in the middle of the segment $a_i a_j$ and the red circle in the middle of the segment 
$a_k a_l$. Then  insert  the blue circle into the red one by joining the cut intervals and
preserving the cyclic ordering. 
 
\item To determine what sign to assign to each new loop obtained in Step~(2), 
think of the segments $[kl]$ and $[ij]$ as oriented lines in  $\bbR^2$; 
in the case where $k=l$, use the tangent
line to the circle at $k=l$, with the anti-clockwise orientation. 
Then, the sign rule is the same as the sign rule for the intersection
$[kl]\cap [ij]$ of oriented lines in $\bbR^2$. 
  
\item  Sum up all circles obtained in Step~\ref{eq:step1}, 
 for all admissible pairs,  with signs  
 as in   Step~\ref{eq:step3}. This sum is our bracket $\{\alpha, \beta\}$.
\end{enumerate}

\begin{prop}\label{P:GoldmanbracketDisk}
 The bracket $\{\alpha, \beta\}$ given by the above procedure is the Goldman 
 bracket of $\alpha$ and $\beta$ in  $H_0^{S^1}(\Lo\DD)$.
\end{prop}

\begin{proof} The proof is similar to the case of two orbifold points. 
By Lemma~\ref{L:bracketreducedorbifold}, $\alpha$ and $\beta$ can, respectively,  be
presented by a red and a blue loop,  as in Figure~\eqref{Fig:loopsmanyorbipoints}. 
These loops lie entirely in the surface  $\DD\setminus\{x_1,\dots, x_r\}$, and their 
Goldman bracket is simply  their usual loop bracket  computed in 
$\DD\setminus\{x_1,\dots, x_r\}$. 
Note that, up to homotopy, we may assume   that  near any orbifold point the red loop 
is confined to a very small neighborhood of the point and is not intersecting the blue loop in
that neighborhood, and  that   $\alpha$ 
and $\beta$ intersect transversally.
\begin{figure} \begin{center}
\includegraphics[scale=0.6]{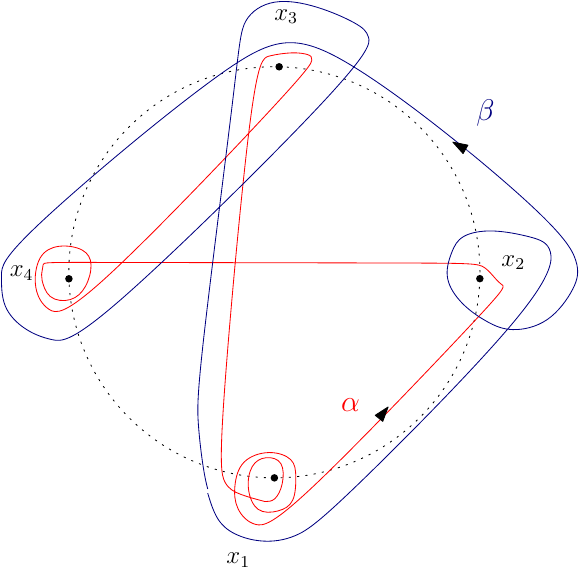}
\caption{A presentation a the red loop $\alpha$ inside a blue loop $\beta$}\label{Fig:loopsmanyorbipoints}
\end{center}
\end{figure}
 Furthermore,  we may also assume that,   when going from the neighborhood of one orbifold 
 point to another, say from $x_k$  to $x_l$, the blue loop intersects  the  
 red loop only when this red  loop is going from a neighborhood of $x_i$ to a
  neighborhood of a distinct $x_j$ in a way that $a_i a_j$ and $a_k a_l$ form an 
  admissible pair; see 
  Figure~\eqref{Fig:loopsmanyorbipoints}. Now, we apply the usual formula for 
  computing the Goldman bracket on   the honest   surface 
  $\DD\setminus\{x_1,\dots, x_r\}$, which yields the result in Step~(4) after 
  quotienting out the relations $a_i^{n_i}=1$, $1\leq i\leq r$.
\end{proof}

\subsubsection{General case}
Now, let $\XX$ be a reduced 2-dimensional
orbifold with finitely many orbifold points $x_1,\dots, x_r$. 
We may assume that $\XX$ is the connected sum of a surface with no orbifold points and a 
genus $0$ surface with $r$ orbifold points, denoted $\DD$. Using
van Kampen we can write a loop in $\XX$ as a sequence  of loops lying alternatively 
in $\DD$ and  $\XX\setminus \DD$. We can then combine the standard definition of the 
Goldman bracket on $\XX\setminus \DD$ and the algorithm given above for $\DD$
to compute the Goldman bracket of free loops in $\XX$.

\begin{ex}
Consider a torus with one orbifold point with isotropy group $\bbZ\slash n \bbZ$. By 
van Kampen, the fundamental group (at a chosen base point) is the 
quotient of the free group on generators $x,y,a$ by the relations $a^n=1$ and $xya=yx$. 
Every free homotopy classes of loops can be presented 
by a word of the form $x^{i_1}y^{j_1}a^{n_1}\dots x^{i_k} y^{j_k} a^{n_k}$.  
 
We find that the loop bracket intertwines the orbifold loop $a$ 
and the non-orbifold loops in a nontrivial manner. For instance, 
one has
 $$ \{x^i, xy^{j}\}= i\sum_{k=1}^j x^{i+1}(ya)^k y^{j-k},  $$
 $$\{x a^n, y^{j} a^m\}= xy^{j}a^{n+m}, \quad  \{x^2 a^{n}, 
 y^{j}a^m\}= xy^{j}a^{n+m}+ x a^{n} xy a^{m}. $$
\end{ex}
  
\begin{ex}
Consider a torus with two orbifolds points of order 3 and 4. By 
van Kampen, the fundamental group (at a chosen base point) is the 
quotient of the free group on generators $x,y,a,b$ by the relations 
$a^3=1$, $b^4=1$ and $xyab=yx$.  Again,  free loops can be presented 
by words of the form 
$x^{i_1}y^{j_1}Q_1(a,b)\dots x^{i_k} y^{j_k} Q_k(a,b)$, where 
$Q_i(a,b)\in \mathbb{Z}/3\mathbb{Z} * \mathbb{Z}/4\mathbb{Z}$ is a 
(noncommutative) word on the loops $a$ and $b$.
 
The string bracket of this orbifold mixes the string product 
of the punctured torus with the string product of the disk with two orbifold
points in a nontrivial way.
For instance, consider the two free loops $\alpha:= x a^2b$ and $\beta:= yab^2$. 
Then, applying our algorithm as in Example~\ref{ex:nonzerodiskwithtwopoints}, 
we get 
$$ \{\alpha, \beta \} = xyab^2a^2b + bxa^2byab^2 -xbyaba^2b = 2xyab^2a^2b - xbyaba^2b, $$ 
where the two terms on the right hand side are  different free loops.
\end{ex}

\subsection{Non-reduced 2-dimensional orbifolds}

Let $\XX$ be an  arbitrary connected 
2-dimensional orbifold with finitely many orbifold points.
It is well known that $\XX$ is a $G$-gerbe over a reduced orbifold $\XX_{\rm red}$, where $G$ 
is a finite group. Let $p\: \XX\to\XX_{\rm red}$ be the structure map.  
By \cite[Proposition 4.6]{Fibrations}, the map
$p\: \XX\to\XX_{\rm red}$ is a weak Serre fibration with fiber $[*/G]$. Therefore, we 
have the homotopy fiber exact sequence
 $$1 \to \pi_2\XX \to \pi_2\XX_{\rm red} \to G \to \pi_1\XX \xrightarrow{p_*}  
 \pi_1\XX_{\rm red} \to 1.$$
Let $K \subseteq G$ be the image of $\pi_2\XX_{\rm red}$ in $G$. So, we have a short exact 
sequence
 $$1  \to G/K \to \pi_1\XX \xrightarrow{p_*} \pi_1\XX_{\rm red} \to 1.$$

\begin{rem}
  The subgroup $K \subseteq G$ is trivial unless $\XX_{\rm red}$ is spherical in the sense of
  \cite{B-N}. In this case, $K$ is a cyclic central subgroup of $G$.
\end{rem} 
 
The map $p\: \XX\to\XX_{\rm red}$ induces a map
 $$p_* \: H_0^{S^1}(\Lo\XX) \to H_0^{S^1}(\Lo\XX_{\rm red}).$$
We also have a map going backwards which is defined as follows.
By  Lemma \ref{L:surjectionpi1}, we have isomorphisms 
 $$H_0^{S^1}(\Lo\XX)\cong k[\Conj(\pi_1\XX)] \text{ and }  
H_0^{S^1}(\Lo\XX_{\rm red})\cong k[\Conj(\pi_1\XX_{\rm red})].$$
We define 
  $$\tau \: k[\Conj(\pi_1\XX_{\rm red})]  \to k[\Conj(\pi_1\XX)]$$
to be the map that sends the conjugacy class $[c] \in  \Conj(\pi_1\XX_{\rm red})$ to
$|K|\sum\pi_*^{-1}(c)$, where $|K|$ is the cardinality of $K$; this 
is easily seen to be independent of the choice of the representative $c$.

\begin{lem}{\label{L:trace}}
We have $p_*(\tau(x))=|G|x$, for every $x \in H_0^{S^1}(\Lo\XX_{\rm red})$.
Furthermore, for any open substack $j \: \UU \to \XX$, the diagram
   $$\xymatrix@M=6pt{ H_0^{S^1}(\UU)   \ar[r]^{j_*}   &    H_0^{S^1}(\XX)   \\   
            H_0^{S^1}(\UU_{\rm red})   \ar[r]_{j_*} \ar[u]^{\tau_{\UU}}
        &   H_0^{S^1}(\XX_{\rm red})   \ar[u]_{\tau_{\XX}} }$$
commutes. 
\end{lem}

\begin{proof}
 The first statement is trivial. The second statement follows from the functoriality 
 of the homotopy fiber exact sequence.
\end{proof}

The following proposition describes the Goldman bracket  $\{-,-\}_{\XX}$ on 
$H_0^{S^1}(\Lo\XX)$ in terms of the Goldman
bracket $\{-,-\}_{\XX_{\rm red}}$ on $H_0^{S^1}(\Lo\XX_{\rm red})$.

\begin{prop}{\label{P:Goldmannonreduced}}
 Let $\alpha, \beta \in H_0^{S^1}(\Lo\XX_{\rm red})$. Then,
  $\{\alpha,\beta\}_{\XX}=\tau\{p_*\alpha,p_*\beta\}_{\XX_{\rm red}}$.
\end{prop}

\begin{proof}
We will only give a sketch of the proof. First of all, note that, by Lemma  \ref{L:trace}
and the fact that $\pi_1\UU \to \pi_1\XX$ is surjective for any 
open substack $\UU \subseteq \XX$ that is the complement of finitely many points in $\XX$
(Lemma \ref{L:bracketreducedorbifold}), 
we may, after replacing $\XX$ by a suitable open substack, 
assume that $X:=\XX_{\rm red}$ is an open surface. In particular, 
$p\: \XX\to X$ is a neutral gerbe for a sheaf of groups $\mcG$ on $X$ that
is locally isomorphic to the constant sheaf $G$. In particular, the fundamental group of
$\XX$ fits into a short exact sequence
   $$1  \to G \to \pi_1\XX \xrightarrow{p_*} \pi_1 X \to 1.$$
Now, using \cite[Proposition 9.3, Lemma 9.4]{BGNX} and functoriality of Gysin maps~\cite[\S~9.2]{BGNX} to compute
the Chas-Sullivan product, we see
that for given loops  $\alpha$ and  $\beta$ in $\XX$, their intersection points are calculated as
follows. For any intersection point $x$ of the loops $p_*\alpha$ and $p_*\beta$ on the honest
surface $X$, we have $|G|$-many intersection points for $\alpha $ and $\beta$ in $\XX$ lying 
above $x$. This is due to the $G$-gerbe structure of $\XX$ at $x$; see 
Remark \ref{R:gerbe} below. The composition
of $\alpha $ and $\beta$ at any of these $|G|$ point  gives a loop in $\XX$
that maps to
$p_*\alpha\circ_{x}p_*\beta$ in $X$; here $\circ_{x}$ stands for loop composition 
at the base point $x$. The resulting $|G|$ loops in $\XX$ are indeed the preimages of 
the loop $p_*\alpha\circ_{x}p_*\beta \in\pi_1 X$ under $p_* \: \pi_1\XX \to \pi_1 X$.
Summing up over all $x \in p_*\alpha\cap p_*\beta$ yields the desired result.
\end{proof}  
  
\begin{rem}\label{R:gerbe}
 Intersecting cycles in a (finite) gerbe can result in intersection numbers that at first may 
 sound odd. The typical situation is when we intersect a point in $[*/G]$ with itself: the 
 intersection is transverse and 
 has multiplicity $|G|$, as $*\times_{[*/G]}*$ is a disjoint union of 
 $|G|$-many points. Similarly, consider $X=\mathbb{R}^2\times [*/G]$, 
 and suppose that we have two transverse lines in $\mathbb{R}^2$, which we view 
 as cycles in $X$.
 By the same reasoning as above,  these two cycles are transverse in
 $X$ and their intersection has multiplicity $|G|$; indeed the intersection points are
 in a  natural bijection with $G$.
\end{rem}  
  

\subsection{A non-orbifold example: quotient stack of a $3$-manifold by $S^1$-action}

We consider here  $2$-dimensional oriented stacks of the form 
 $[M/S^1]$, where $M$ is a $3$-manifold $M$ endowed with a smooth circle action. 
Such a stack is differentiable of dimension 2 and (in general) is not an orbifold. 

Since  $[M/S^1]$  is $2$-dimensional, it has a Goldman bracket defined 
on the equivariant degree $0$ homology $H_0^{S^1}([M/S^1],\mathbb{Z})$,
which is the same as the free abelian group 
spanned by the free homotopy classes of  loops in $[M/S^1]$ (see Lemma~\ref{L:surjectionpi1}).

Compact oriented manifolds equipped with a $S^1$-action are classified 
in~\cite[Theorem 1]{Ra}. In the case where there are no exceptional orbits, such a manifold
is equivariantly diffeomorphic to a manifold $M_{ g, h, t}$ obtained as a quotient
\begin{equation} \label{eq:Mghtquotient} 
M_{ g, h, t} := \Big(\Sigma_{g, h, t} \times S^{1} \Big)/ \simeq, 
\end{equation} 
where $\Sigma_{g, h, t}$ is a compact oriented surface 
of index $g$, with $h+t$ boundary components, $h\geq 1$. The equivalent  
relation $\simeq$ works as follows. For each of the $h$ boundary circles, the (trivial) 
circle bundle over it is collapsed onto the base circle
via the projection map. Over each of  the remaining $t$ circles, the fibers
of the corresponding (trivial) circle bundle are quotiented out by the antipodal action. 

It can be shown (see~\cite{Ra}) that $M_{ g, h, t}$ is a connected sum 
 \begin{equation} \label{eq:Mghtdecomp} M_{g,h,t} \cong S^3\# (S^2\times S^1)_1\#\cdots \# 
 (S^2\times S^1)_{2g+h-1} \#(\mathbb{C}P^1)_1 \#\cdots \#(\mathbb{C}P^1)_t.\end{equation}

From now on, we will consider, 
for simplicity\footnote{one can also do computations for arbitrary $t$  and obtain 
the same result as Proposition~\ref{P:MmodoutbyS1}.},
the case $t=0$. We write $\XX_{g,h}:= [M_{g,h,0}/S^1]$ for the induced quotient stack. 
Note that $[M_{g,h,0}/S^1]$ is not
an orbifold since the locus $F$ of the $S^1$-fixed points of $M_{g,h,0}$ is not empty. 
Our goal, now, is to relate the Goldman Lie algebra of  $\XX_{g,h}$ 
with the interior of the surface $\Sigma_{g,h,0}$, see Proposition~\ref{P:MmodoutbyS1}.

We first compute the homology of the stack $\XX_{g,h}$.
\begin{prop}
 One has $H_0(\XX_{g,h})\cong \mathbb{Z}$ and 
 $$H_1(\XX_{g,h})\cong \mathbb{Z}^{2g+h-1} , \quad H_{i\geq 2}(\XX_{g,h})\cong \mathbb{Z}^{h}.$$
\end{prop}
\begin{proof}
  From~\eqref{eq:Mghtdecomp}, one  sees that 
  $H_0(M_{g,h,0})\cong \mathbb{Z} \cong H_3(M_{g,h,0})$ 
  and 
 $$H_1(M_{g,h,0})\cong \mathbb{Z}^{2g+h-1} \cong H_2(M_{g,h,0}), 
 \quad H_{i>3}(M_{g,h,0})\cong 0.$$
 We now apply the Gysin sequence of the $S^1$-principal bundle $M_{g,h,0}\to \XX_{g,h}$
 and the Leray-Serre spectral sequence of $\XX_{g,h}\to [*/S^{1}]$ to compute the homology of
 $\XX_{g,h}$ (which is the $S^1$-equivariant homology of $M_{g,h,0}$).
 
 To do this we need the operator $D: H_n(M_{g,h,0}) \to H_{n+1}(M_{g,h,0})$ 
 induced by the $S^1$-action.  Since $M_{g,h,0}$ is path connected, 
 we can take any point $x$ as a generator of  $H_0(M_{g,h,0})$.
 Taking $x$ to be $S^1$-invariant, we see that that $D$ is zero in degree $0$,
 and thus in degree $3$ by Poincar\'e duality. 
 A set of generators of $H_1(M_{g,h,0})$ is given by the class 
 (in the quotient $M_{g,h,0}=\Big(\Sigma_{g, h, t} \times S^{1} \Big)/ \simeq$) 
 of products of the standard loops $\gamma_1,\dots,\gamma_{2g}$, $a_1,\dots, a_{h-1}$ of 
 $\Sigma_{g, h, 0}$ with $\{1\}\in S^1$, 
 where the $\gamma_i$ are the $2g$ generators  corresponding to the genus of the surface 
 and $a_1,\dots, a_h$ are the boundary class of  $\Sigma_{g, h, t}$. 
 Note that in  $M_{g,h,0}$, 
 we can choose to have the generators satisfying 
 $$a_h\times \{1\} = \sum_{i=1}^{h-1} a_i\times\{1\} +\sum_{j=1}^{2g}\gamma_{i}\times\{1\}.$$
 The operator $D$ maps the generators $\gamma_i\times\{1\}$ to $\gamma_i\times S^1$,
 which are linearly independent in $H_2(M_{g,h,0})$. 
 However, since $S^1$ acts trivially on $a_i\times \{1\}$, 
 the generator $D$ maps it to $0$. 
 
 The homology of $\XX_{g,h}$ being the $S^1$-equivariant homology of $M_{g,h,0}$, 
 it is also the homology of the bicomplex $ (C_*(M_{g,h,0})[u], +u^{-1}D)$
 where $|u|=2$ and $d:C_*(M_{g,h,0})\to C_{*-1}(M_{g,h,0})$ is 
 the usual singular differential. Here we use the standard bicomplex 
 to compute equivariant homology (note that our $u$ is sometimes written as $u^{-1}$ 
 in the literature). 
 
 The spectral sequence associated to this bicomplex (which is a special case
 of Leray spectral sequence in that case) has $E^1$ page given by 
 $E^1_{p,q} = H_{p-q}(M_{g,h,0}) u^q$, and the induced differential $d^1$ 
 is given by $d^1(x u^i)= D(x) u^{i-1}$ for $x\in H_{p-q}(M_{g,h,0})$.   
 From our computation of $D$, we find that 
 $E^2_{0,0}\cong  \mathbb{Z} $, $E^2_{0,1} = 0$ and 
 $E^2_{1,0}  \cong \mathbb{Z}^{2g+h-1}$. All others $E^2_{p,q}$ are zero except
 $$E^2_{n+1,n-1}\cong  \mathbb{Z}^{h-1}, \quad E^2_{n,n}\cong  \mathbb{Z},\quad
 E^2_{n+1,n}\cong  \mathbb{Z}^{h-1}, \quad E^2_{n+2,n-1}\cong  \mathbb{Z},$$ 
 where $n\geq 1$. 
 The only possible nontrivial higher differentials in the spectral sequence are the 
 $u$-linear maps 
   $$d^2: H_0(M_{g,h,0})u^{i} \to H_3(M_{g,h,0})u^{i-2}.$$
 By the Gysin sequence associated to the map $M_{g,h,0}\to \XX_{g,h}$, 
 we see right away that $H_{n}(\XX_{g,h}) \cong H_{n+2}(\XX_{g,h})$ for all $n\geq 3$, 
 since $H_{i>3}(M_{g,h,0})=0$. In particular, 
 $$ \bigoplus_{p+q=2}E^{\infty}_{p,q}=\bigoplus_{p+q=2} E^2_{p,q} = 
 \mathbb{Z}^{h-1}\oplus \mathbb{Z} \text{ and }  \bigoplus_{p+q=1}E^{\infty}_{p,q}=
 \bigoplus_{p+q=1} E^2_{p,q} = \mathbb{Z}^{2g+h-1}. $$  
 We are left to compute $d^2: H_0(M_{g,h,0})u^2 \to H_3(M_{g,h,0})$ 
 to find the degree $3$ and degree $4$ homology classes in the abutment 
 $E^\infty_{*,*}$ of the spectral sequence.  
 Again, since $H_{i>3}(M_{g,h,0})=0$, the Gysin long exact sequence 
 gives rise to the long exact sequence 
 \begin{multline*} H_4(\XX_{g,h}) \hookrightarrow H_2(\XX_{g,h}) 
 \stackrel{T}\to H_3(M_{g,h,0})\cong \mathbb{Z} \to 
 H_3(\XX_{g,h})\\ \to H_1(\XX_{g,h}) \stackrel{T}\to H_2(M_{g,h,0}) \to 
 H_2(\XX_{g,h}) \twoheadrightarrow H_0(\XX_{g,h})\cong \mathbb{Z}
 \end{multline*}
Using what we have already computed, we see that the last line boils down 
(on the associated graded induced by the filtration of the bicomplex in) to an  exact sequence
 $$\to \mathbb{Z}^{2g} \oplus \mathbb{Z}^{h-1}\stackrel{T}\to\mathbb{Z}^{2g}\oplus 
 \mathbb{Z}^{h-1}  \to  \mathbb{Z}^{h-1}\oplus \mathbb{Z} \twoheadrightarrow  \mathbb{Z} .$$ 
 By Lemma~\ref{L:XX=M}, the map 
   $$\mathbb{Z}^{2g+h-1}\stackrel{T}\to\mathbb{Z}^{2g+h-1}$$ 
 can be identified with the map  
  $$H_1(M_{g,h,0}) \stackrel{D}\to H_2(M_{g,h,0}),$$ 
 hence is just the trivial projection 
 $$\mathbb{Z}^{2g}\oplus \mathbb{Z}^{h-1} \twoheadrightarrow \mathbb{Z}^{2g} 
 \hookrightarrow \mathbb{Z}^{2g}\oplus \mathbb{Z}^{h-1}$$ 
 on $\mathbb{Z}^{2g}$. 
 Similarly, since the map $H_2(M_{g,h,0})\to H_2(\XX_{g,h})$ is induced on the 
 associated graded to the filtration by the map  
 $$H_2(M_{g,h,0}) = E^1_{2,0} \hookrightarrow E^1_{2,0}/ d^1(E^1_{2,1})= 
 E^2_{2,0}\hookrightarrow  E^2_{2,0} \oplus E_{1,1},$$ 
 and that $D$ is trivial in degrees $0$ and $3$, 
 we see that $H_2(\XX_{g,h}) \stackrel{T}\to H_3(M_{g,h,0}) $ is the zero map as well.
 This proves that  $H_4(\XX_{g,h}) \cong H_2(\XX_{g,h})$,  and that  
 $d^2: H_0(M_{g,h,0})u^2 \to H_3(M_{g,h,0})$ is zero and 
  $$\bigoplus_{p+q=3}E^{\infty}_{p,q} \cong \mathbb{Z}^{h-1}\oplus \mathbb{Z}.$$
This gives us the associated graded to the homology of $\XX_{g,h}$; since it consists only of finitely generated free modules, there is no nontrivial extensions  and this is isomorphic to
 $H_3(\XX_{g,h})$.
 \end{proof}

The action of $S^1$ on $M_{g,h,0}\setminus F$ is free and the quotient stack 
$[(M_{g,h,0}\setminus F)/S^1]$ is isomorphic to the surface $(\Sigma_{g, h, 0})^0$,  
the interior of $\Sigma_{g, h, 0}$. We thus have an embedding of differentiable 
stacks 
$$(\Sigma_{g, h, 0})^0 \cong [(M_{g,h,0}\setminus F)/S^1] 
\ \hookrightarrow \ [M_{g,h,0}/S^1]= \XX_{g, h}.$$

\begin{lem}\label{L:XX=M}
 The canonical epimorphism $  M_{g,h,0} \to \XX_{g,h}$ and the canonical embedding  
 $(\Sigma_{g, h, 0})^0 \to  \XX_{g, h}$ induce isomorphisms  on $\pi_1$.
\end{lem}

\begin{proof}
 Note that all the stacks involved are pathwise connected. 
 The long exact sequence of $S^1$-fibrations yields a commutative diagram of exact sequences
 \[\xymatrix{ \{*\} & \pi_1(\XX_{g,h}) \ar[l] & \pi_1(M_{g,h,0}) \ar[l] 
    & \mathbb{Z} \ar[l] \\
  \{*\} & \pi_1((\Sigma_{g, h, 0})^0) \ar[u] \ar[l] & \pi_1(M_{g,h,0}\setminus F) 
  \ar[l] \ar[u]   & \mathbb{Z} \ar@{_{(}->}[l] \ar@{=}[u]  .}  \]
 The lower sequence is split exact since $M_{g,h,0} \setminus F 
 \cong (\Sigma_{g, h, 0})^0 \times S^1$.
  The top right map is the zero map since it is loop homotopic to the class of 
  $\{x\}\times S^1$,  which is just a point when $x$ belongs to a $h$ boundary component.  
  
  Furthermore, the middle vertical  map is surjective by definition of $M_{g,h}$:
  it kills the generator coming from the $S^1$ factor. It remains to prove that the induced map 
  $$ \pi_1((\Sigma_{g, h, 0})^0)  \cong \pi_1(M_{g,h,0}\setminus F)/\mathbb{Z} 
    \to  \pi_1(M_{g,h,0})$$ 
  is an isomorphism.  This isomorphism follows from  van Kampen  and the 
  presentation~\eqref{eq:Mghtdecomp} which identifies  both fundamental groups with the free group
\begin{eqnarray*} F(\gamma_1,\dots, 
    \gamma_{2g}, a_1,\dots, a_h)/(a_h=\gamma_1\cdots\gamma_{2g}\cdot a_1\cdots a_{h-1}) \\ \cong
    F(\gamma_1,\dots, 
\gamma_{2g}, a_1,\dots, a_{h-1}) 
\end{eqnarray*}  
on the generators $\gamma_i$ and $a_j$. 
\end{proof}

 From the previous lemma, we deduce that the stack embedding 
 $(\Sigma_{g, h, 0})^0\hookrightarrow \XX_{g,h}$ induces an isomorphism of 
 Goldman Lie algebras (after passing to free loops).
 
\begin{prop}\label{P:MmodoutbyS1}
 The induced map $H_0^{S^1}(L(\Sigma_{g, h, 0})^0) \to H_0^{S^1}(\Lo\XX_{g,h})$
 is an isomorphism of Lie algebras.
\end{prop}

\begin{proof}
 When $\XX$ is a pathwise connected stack, we have an natural identification of $H_0(\Lo\XX)$ 
with the free module on the set of conjugacy classes in $\pi_1(\XX)$  
(Lemma~\ref{L:surjectionpi1}). 
The proposition then follows from   Lemma~\ref{L:XX=M} combined with 
Proposition~\ref{P:functoriality}.
\end{proof}

In particular, we see that the Goldman bracket on $\XX_{g,h}$ is highly nontrivial.  
By van Kampen and  Lemma~\ref{L:XX=M}, the
Goldman Lie algebra $H_0^{S^1}(\Lo\XX_{g,h})$ of $\XX_{g,h}$ is isomorphic to 
\begin{eqnarray*}   
\mathbb{Z}\big[\Conj\big(F(\gamma_1,\dots, 
\gamma_{2g}, a_1,\dots, a_h)/(a_h=\gamma_1\cdots\gamma_{2g}\cdot a_1\cdots a_{h-1})\big)
\big] \\ \cong \mathbb{Z}\big[\Conj\big(F(\gamma_1,\dots, 
\gamma_{2g}, a_1,\dots, a_h)  \big) \big],
\end{eqnarray*}
where $F(\gamma_1,\dots, \gamma_{2g}, a_1,\dots, a_{h-1})$ is the free group on the
given generators and
$\Conj(H)$ stands for the set of conjugacy classes of a group $H$. The Lie bracket 
can be computed purely in terms of geometric intersections of the words in the generators.
In particular, the Lie bracket of the images of the generators $\gamma_i$ and $\gamma_{j}$ seen
as free loops in $\XX_{g,h}$ is a nontrivial free loop in general. 

\medskip

It would certainly be interesting to see if there are higher dimensional 
non-zero brackets on $H_{\bullet}^{S^1}(\Lo\XX_{g,h})$,  but the computations are 
much more involved in that case.


\providecommand{\bysame}{\leavevmode\hbox
to3em{\hrulefill}\thinspace}
\providecommand{\MR}{\relax\ifhmode\unskip\space\fi MR }
\providecommand{\MRhref}[2]{%
  \href{http://www.ams.org/mathscinet-getitem?mr=#1}{#2}
} \providecommand{\href}[2]{#2}

\end{document}